\documentclass{amsart}

\usepackage[utf8]{inputenc}
\usepackage{graphicx,xcolor}
\usepackage{subcaption}
\usepackage{mathrsfs,amsmath,amssymb,amsthm,amsfonts,mathtools,thmtools,dsfont}
\usepackage{hyperref,cleveref}
\usepackage{enumerate,subcaption,booktabs}

\usepackage{tikz}
\usetikzlibrary{calc,angles,quotes,patterns}
\usetikzlibrary{patterns.meta}
\usetikzlibrary{decorations.pathreplacing,ext.transformations.mirror}
\usepackage{pgfplots}
\pgfplotsset{compat=1.18}
\usepgfplotslibrary{external}

\usepackage[top=3cm,bottom=3cm,left=3cm,right=3cm,marginparwidth=2.5cm]{geometry}
\usepackage[activate={true,nocompatibility},final,tracking=true,kerning=true,spacing=true,factor=1100,stretch=10,shrink=10]{microtype}
\microtypecontext{spacing=nonfrench}

\newtheorem{theorem}{Theorem}
\declaretheorem[sibling=theorem,name=Lemma]{lemma}
\declaretheorem[sibling=theorem,name=Corollary]{corollary}
\declaretheorem[sibling=theorem,name=Proposition]{proposition}
\theoremstyle{definition}

\theoremstyle{remark}
\newtheorem*{remark}{Remark}

\newcommand{\C}{\mathbb C}
\newcommand{\R}{\mathbb R}
\newcommand{\N}{\mathbb N}
\newcommand{\Z}{\mathbb Z}
\newcommand{\abs}[1]{\left\lvert #1 \right\rvert}

\renewcommand{\d}{\mathrm d}
\renewcommand{\epsilon}{\varepsilon}
\newcommand{\rightplane}{\mathbb{H}}
\newcommand{\rayon}{\rho_a(x)}
\renewcommand{\L}[1]{\mathcal L_{#1}(a;x)}
\newcommand{\fgd}{\phi(a;x)}

\DeclareMathOperator{\arcsinh}{arcsinh}

\DeclareMathOperator{\sgn}{sgn}

\usepackage{biblatex}
\bibliography{references}
\usepackage[acronym]{glossaries-extra}
\setabbreviationstyle[acronym]{long-short}
\newacronym{erw}{ERW}{Elephant Random Walk}
\newacronym{ldp}{LDP}{Large Deviations Principle}
\newacronym{mgf}{MGF}{Moment Generating Function}
\newacronym{lt}{Laplace transform}{Laplace transform}
\glsdisablehyper

\title[Laplace transform of the Elephant Random Walk]
{Asymptotics for the Laplace transform of the Elephant Random Walk via Schwarz-Christoffel mappings}
\date{July 16, 2026}
\author{ Juan Esaul González Rangel}

\begin{document}
\begin{abstract}
     We derive the long time asymptotic behavior of the Laplace transform of the Elephant Random Walk (ERW). The ERW is a generalization of the Simple Random Walk in which the past influences the future evolution of the process according to some memory parameter. Its transition kernel is neither space nor time homogeneous and its transition probability at the $n$th step can be written as a function of $n$ and the position of the walk at time $n-1$. A recurrence relation for its Laplace transform allows one to find a generating function governed by a transport-type PDE whose solution can be written in terms of a Schwarz–Christoffel mapping, allowing a precise analytic control. This enables to apply singularity analysis methods to estimate the Laplace transforms. 
\end{abstract}
\maketitle
\section{Introduction and main results}

 The \gls{erw} was introduced by Schütz and Trimper \cite{elephants_never_forget} as a generalization of the Simple Random Walk where the increments depend on the full past of the process. Their goal was to study how memory effects can lead to anomalous diffusion. Despite its simple definition, its rich properties have motivated a broad research literature around, among others, its path-wise convergence \cite{baur2016elephant}, limit distribution \cite{guerin2025limit_superdiffusive}, generalizations and connection to other models \cite{podder2025,multidimensional_elephant}. 

To define an \gls{erw}, take $p,q\in[0,1]$ and let $X_1$ be a Rademacher random variable taking the value 1 with probability $q$ and $-1$ with probability $1-q$. For $n>1$, let $U(n)$ be a uniform random variable in $\{1,2,\ldots,n-1\}$. Then, $X_n$ is defined as $X_{U(n)}$ with probability $p$ and as $-X_{U(n)}$ with probability $1-p$. The \gls{erw} at time $n$ is given by $S_n = \sum_{j=1}^n X_j$. 

The value $q$ is called the first-step parameter and $p$ is called the memory parameter, it is interpreted as the willingness of the elephant to repeat its past. It will be convenient to write $p$ as $p=\frac{1+a}{2}$, with $a\in[-1,1]$. 

Let $\mathscr F = (\mathscr F_n, n\in \N)$ be the natural filtration of $S$. Conditioning the $(n+1)$th increment on the filtration at time $n$ yields
\begin{align}
    \mathbb P( X_{n+1} = 1 | \mathscr F_n) &= \frac{\# \{ X_j, j \le n : X_j =1 \}}{n} p + \frac{ \# \{ X_j, j \le n : X_j = -1\} }{n}(1-p) \nonumber\\
    &= \frac{n+S_n}{2n}p + \frac{n-S_n}{2n}(1-p) = \frac12\left( 1 + a\frac{S_n}{n} \right). \label{eq:recursion_probability}
\end{align}
Two special cases arise as particular instances of \eqref{eq:recursion_probability}. When $a = 0$, the past steps do not influence the behavior of the walk and then the \gls{erw} coincides with the Simple Random Walk, a process whose $n$th step independently takes the value $1$ or $-1$ with probability $1/2$. When $a=1$, the \gls{erw} repeats the sampled step at every time $n$, so every step coincides with $X_1$ and then it exhibits a ballistic behavior depending only on the first-step parameter $q$. 

As mentioned above, since its introduction, the \gls{erw} has been extensively studied. It was conceived as a model exhibiting a diffusive--superdiffusive phase transition induced solely by long-range memory. Specifically, when $a<1/2$, the variance of the walk grows linearly with time $n$, whereas for $a>1/2$, it grows at the rate $n^{2a}$. In the critical regime, $a=1/2$, the variance grows at the rate $n\log n$. This phase transition was first observed in the seminal paper by Schütz and Trimper \cite{elephants_never_forget}. Its connections to Polya type urns were used to establish the convergence to a Gaussian process in the diffusive case and non-Gaussian in the superdiffusive case \cite{baur2016elephant}. Details of its asymptotic behavior were obtained by means of a martingale approach \cite{bercu2018martingale}, while some other approaches were used to explore the properties of its limit distribution \cite{guerin2023fixed_point,guerin2025limit_superdiffusive}. Some generalizations exist, modifying its memory mechanism \cite{podder2025} or its state space \cite{multidimensional_elephant}. 
Particularly, the \gls{erw} is known to satisfy a \gls{ldp} with a rate function depending on $a$ \cite{franchini2023large,franchini2015large}. 

The main contribution of this paper is the determination of the precise long-time asymptotic behavior of the Laplace transform of the \gls{erw}. The key observation is that the generating function of Laplace transforms can be represented through the inverse of a Schwarz--Christoffel integral. This representation provides a precise description of its analytic continuation and singularities, allowing the asymptotic behavior of the Laplace transform to be derived by singularity analysis.

Knowing the behavior of the Laplace transform is interesting in its own right and may also be useful for deriving other results. In particular, in \Cref{last_subsection} we show how our result can be used to recover the known \gls{ldp} for the \gls{erw}. Moreover, it is reasonable to expect that it may provide a starting point for obtaining further results that are, to the best of the author's knowledge, still open, such as sharp large deviation estimates in the spirit of Bahadur and Ranga Rao in a \cite{bahadur1960deviations} in a similar fashion as the results for the number of descents of a random permutation by Bercu, Bonnefont et Richou \cite{bercu2024sharp} or a local limit theorem.

For the following, let $\L{n}$ denote the Laplace transform of $S_n$ (equivalently, the moment generating function at negative arguments): $$\mathcal L_n(q,a;x) \coloneqq \mathbb E[ e^{-xS_n} ].$$ The dependence of $\mathcal L_n$ on $q$ will not be written unless necessary.

An interesting property is that if an \gls{erw} $S$ has first step parameter $q$ and memory parameter $p$, then $Z=(-S_n, n \ge 0)$ is an \gls{erw} with the same memory parameter and first step parameter $1-q$. Indeed, it is immediate that $\mathbb P\left( Z_1 = 1 \right) = \mathbb P \left( S_1 = -1 \right) = 1-q$ and by \cref{eq:recursion_probability}:
\begin{align*}
    \mathbb P\left( Z_{n+1} = Z_n + 1 \mid \mathscr F_n \right) &= \mathbb P \left( S_{n+1} = S_n - 1 \mid \mathscr F_n \right) = \frac12\left( 1 - a\frac{S_n}{n} \right) = \frac12\left( 1 + a\frac{Z_n}{n} \right).
\end{align*}
This implies that 
\begin{equation}
    \mathcal L_n(q,a;-x) = \mathcal L_n(1-q,a;x). \label{eq:even}
\end{equation}
The identity above allows one to study $\L{n}$ only for positive values of $x$.

In the cases $a=0$ (the Simple Random Walk) and $a=1$, it is immediate to verify \cref{eq:even} by finding the Laplace transform of $S_n$ exactly:
\begin{align}
    \mathcal L_n(q,0;x) &= \left( qe^{-x} + (1-q)e^{x} \right)\cosh^{n-1}(x), \qquad \mathcal L_n(q,1;x) = q e^{-nx} + (1-q)e^{nx}. \label{eq:zero_and_one}
\end{align}

The main result of this paper is stated in \Cref{result}, which provides an asymptotic expression for $\L{n}$ as $n \to \infty$. It is obtained through a detailed analysis of the generating function $$ L_a(x,z) \coloneqq \sum_{n\ge 0} \L{n+1} z^n,$$ yielding an explicit expression and a nontrivial analytic continuation to a sufficiently large complex domain. This analysis contains most of the technical work of the paper and enables the application of singularity analysis techniques.

Two results stemming from this analysis are of independent interest. \Cref{form_solution} provides an explicit expression for $L_a(x,z)$ involving a Schwarz--Christoffel mapping, while \Cref{analicity} identifies a complex domain in which the map $z \mapsto L_a(x,z)$ is analytic for fixed values of $x$ and $a$. Together, these results establish the analytic framework required to derive the asymptotic behavior of \Cref{result}. They may also be useful for further applications.

Finally, \Cref{non_analytic}, which characterizes the failure of analyticity at the origin of the function governing the exponential growth rate of the Laplace transform, constitutes the fourth main result of the paper. It reveals distinct behaviors for positive and negative values of $a$, as well as for the regimes $a<1/2$ and $a\ge1/2$. Moreover, additional singular regimes occur for values of $a$ of the form $a=1/n$, where $n$ is any positive integer.

Before stating the main result, it is convenient to define the function that determines the exponential growth rate of the Laplace transform. For every $a\in [-1,1]\setminus \{0\}$ fixed, define $\phi(a;\cdot): \R \to \R_+$ as 
\begin{equation} \displaystyle
    \fgd \coloneqq \left\{ \begin{array}{cc}
        \frac{\sinh^{-\frac1a}(\abs{x})}{-\frac{1}{a}\int_0^{\abs{x}} \sinh^{-1-\frac{1}{a}} (s) \d s}, & a \in [-1,0), \\
        \frac{\sinh^{-\frac1a}(\abs{x})}{ \frac1a\int_{\abs{x}}^\infty \sinh^{-1-\frac1a} (s)\d s }, & a \in (0,1].
    \end{array} \right. \label{eq:definition_fgd}
    \end{equation} 

\begin{theorem} \label{result}
    The following estimate holds as $n$ goes to $\infty$:
    \begin{align} \displaystyle
        \L{n} = \left\{ \begin{array}{cc}
        \left( \fgd \right)^n(1 + o(1)), &\text{if $a\in [-1,0)$,} \\
        2\left( \frac{1-q}{e^{-2x}+1} + \frac{q}{e^{2x} +1} \right)\left(\cosh(x)\right)^n(1 + o(1)),  &\text{ if } a = 0,\\
        2\left( \frac{1-q}{a+1}\mathds 1_{x>0} + \frac{q}{a+1} \mathds 1_{x<0} \right)\left( \fgd \right)^n(1 + o(1)),  &\text{if $a\in(0,1]$.}
        \end{array} \right.  \label{eq:asymptotic_laplace} 
    \end{align}
\end{theorem}
Notice that the exponential term depends on $a$ but not on $q$. Furthermore the dependence on $q$ appears only for $a \in [0,1]$. In the extreme case $a=-1$, $S_2=0$ with probability 1, and the first step parameter no longer influences the behavior of the \gls{erw} with memory parameter $0$ after the time $n=2$. Also, for every value of $a$, the dependence on $x$ is only through $\abs{x}$, in concordance with \cref{eq:even}.

An \gls{ldp}, including an expression of the limit cumulant generating function for a large class of models containing one which is equivalent to the \gls{erw} is found in \cite{franchini2023large}. However, to the best of the author’s knowledge, a precise asymptotic expression for the Laplace transform has not explicitly appeared in the literature. This \gls{ldp} can also be found as a corollary to \Cref{result}.

\begin{corollary} \label{large_deviations} For $a\in[-1,1)$, the \gls{erw} satisfies an \gls{ldp} with good rate function $\Lambda^*(x)$ defined as $$ \Lambda^*(x) \coloneqq \sup_{t}\left\{ tx - \log \phi(a;t) \right\},$$ with $\phi(a;t)$ as defined in \cref{eq:definition_fgd}. 
\end{corollary}
To see the proof of \Cref{large_deviations} and its equivalence with the known \gls{ldp} in \cite{franchini2023large}, see \Cref{last_subsection}.

The \gls{erw} is known to have a phase transition at the value $a=1/2$, when $a<1/2$ the walk is diffusive, meaning that its variance grows linearly with $n$, while for values $a\ge1/2$ it becomes super-diffusive \cite{elephants_never_forget}, i.e. \[\lim_{n\to \infty} \frac{Var(S_n)}{n} = \infty. \] Interestingly enough, at first glance the Laplace transform asymptotics does not seem to reflect this behavior and only shows a dependence on the sign of $a$. However, a finer analysis reveals a true difference between the cases $a<1/2$ and $a\ge1/2$: It turns out that $x \mapsto \phi(a;x)$ is of class at least $\mathcal C^2$ when $a<1/2$, but this is not always the case for $a\ge1/2$. Furthermore, when $a>0$, the $\phi(a;x)$ is never analytic at $x=0$, and the number of derivatives it admits depends on $a$. This result was mentioned but not proven in \cite{franchini2023large} and it is stated in \Cref{non_analytic}.

\begin{proposition} \label{non_analytic}
The function $x \mapsto \phi(a;x)$ is analytic at $0$ for $a \in [-1,0)$ and non-analytic for $a \in (0,1]$. Moreover, for $a>0$, it is of class $C^{\lfloor 1/a \rfloor -1}$, but not smoother. The first singular term in its expansion at 0 is of order $\abs{x}^{1/a}$ when $1/a$ is not an even integer, and of order $x^{1/a}\ln(\abs{x})$ when $1/a$ is an even integer.
\end{proposition}

The asymptotic of $\L{n}$ is obtained by estimating the coefficients of the generating function \begin{equation} L_a(x,z) \coloneqq \sum_{n\ge 0}  \L{n+1} z^n .\label{eq:generating_function}\end{equation} Several well-known techniques in combinatorics allow to do so. In particular, the one used in this paper (see \Cref{flajolet}) requires the generating function to have a singularity at $z=1$ and to be analytic in a domain of the form 
\begin{equation*}
    \Delta(\epsilon,\theta) \coloneqq \{z : \abs{z} \le 1 + \epsilon,  \abs{\mathrm{Arg}(z-1)} \ge \theta \}.
\end{equation*}
 If $F$ is analytic in some $\Delta(\epsilon,\theta)$ with a singularity at $z=1$ and near the singularity it satisfies an asymptotic behavior $F(z) \sim (1-z)^{\alpha}$, then $[z^n]F(z) \sim n^{-\alpha-1}/\Gamma(-\alpha)$ as $n\to\infty$. More details of this method can be found in \cite{flajolet1990singularity}, see also \Cref{flajolet}. Notice that, by normalization, it is possible to locate the singularity in any point of the complex plane, so to estimate the coefficients of a certain function $F$, it suffices to prove that it is meromorphic in a domain of form $\Delta(\epsilon,\theta)$ with a single singularity at the border (see \cite{flajolet1990singularity}). This will happen to be the case of $L_a(x,z)$, whose exact form is given by \Cref{form_solution}.

\begin{proposition}
\label{form_solution} For every $a\in [-1,1)\setminus \{0\}$, there exists $r_a(x)>0$ such that for all $x\in \R$, and $z \in (-r_a(x), r_a(x))$, the generating function $L_a(x,z)$ defined in \eqref{eq:generating_function} is equal to
\begin{equation} \label{eq:equation_for_Laplace}
        L_a(x,z) = \sinh^{-\frac{1}a} (x)A\left(k^{-1}\left( k(\sinh(x)) - z\sinh^{-\frac1a}(x)  \right) \right),
\end{equation}
    where $A$ is the algebraic function
\begin{align}
     A(t) &\coloneqq \left( (1-2q)t +\sqrt{1+t^2} \right)t^\frac1a,\label{eq:algebraic_function} \\
     \intertext{and $k^{-1}$ is the inverse under composition of the strictly monotone function $k:\R^+ \to \R^+$, defined as}
     k(t) &\coloneqq \left\{ \begin{array}{cc}
          -\frac1{a} \int_0^t \frac{\d s}{s^{\frac{1+a}{a}}\sqrt{1+s^2}}, & a \in [-1,0), \\
          \frac1{a} \int_t^\infty \frac{\d s}{s^{\frac{1+a}{a}}\sqrt{1+s^2}},  & a \in (0,1).
     \end{array} \right. \label{eq:definition_k}
 \end{align}
\end{proposition}

The expression found for $L_a(x,z)$ is reminiscent to several functions that appear in probabilistic combinatorics and in the study of the \gls{erw}. In \cite{guerin2025elephantpolynomials}, an analogous function arises as the moment generating function of the limiting distribution of the \gls{erw} in the superdiffusive regime. Related expressions also appear in \cite{bercu2024sharp} and \cite{bryc2009large} in connection with the distribution of the number of descents in a random permutation and the number of leaves in random trees, respectively. Notably, all of these models, including the \gls{erw}, admit formulations in terms of urn processes. In this paper, we are able to derive this explicit expression of $L_a(x,z)$ for all $a\in[-1,1]$. More precisely, we determine its domain of analyticity in the variable $z$ and describe the location and nature of its singularities for every fixed value of $a$. The main challenge lies in obtaining a unified treatment for the full range of values of $a$, which exhibit markedly different analytic behaviors.

The function $L_a(x,z)$ has a number of remarkable properties. The auxiliary function $k(t)$, defined in \eqref{eq:definition_k}, is closely related to several classical special functions: it can be expressed in terms of hypergeometric and beta functions (see \cref{eq:hypergeometric,eq:beta}, respectively), and it is also a particular instance of a Schwarz--Christoffel integral (see \Cref{sc}). Historically, Schwarz--Christoffel mappings arose as some of the only explicit realizations of the Riemann Mapping Theorem. This connection not only yields explicit representations of $L_a(x,z)$ in terms of well-studied functions, but also allows one to determine its optimal domain of analyticity and describe its singularities in considerable generality. Although related special functions have previously appeared in the study of urn models (see, e.g., \cite{flajolet2005urns}), to the best of the author's knowledge this is the first time that Schwarz--Christoffel techniques have been used in the context of the \gls{erw}.

The function $A$ is affine in $q$, which allows us to decompose $L_a(x,z)$ into the superposition of a term proportional to $q$ and a term proportional to $1-q$:
$$ A(t) = (1-q)t^{\frac{a+1}{a}} - qt^{\frac{a+1}{a}} + t^\frac1a\sqrt{1+t^2}.$$
For particular values of $a$, the expression for $L_a(x,z)$ simplifies considerably. Notably, when $a=-1$, one has $k(t)=\arcsinh(t)$ and
$$L_a(x,z)=\sinh(x)\left(1-2q+\coth\left(x-z\sinh(x)\right)\right).$$ When $a=1/n$ with $n\in\mathbb Z\setminus\{0\}$, repeated integration by parts yields recursive expressions for $k(z)$ (see \Cref{last_subsection}).

\Cref{form_solution} is found using the first-order transport-type partial differential equation
\begin{equation}
\left(1-z\cosh(x)\right)\partial_zL_a(x,z)
= a\sinh(x)\partial_xL_a(x,z)+\cosh(x)L_a(x,z),
\label{eq:eqdiff}
\end{equation}
which characterizes the generating function $L_a(x,z)$. Solving this equation, with an appropiate initial condition, leads to the explicit representation \eqref{eq:equation_for_Laplace}. Thus, the problem of determining the generating function is reduced to the analysis of this transport equation, whose characteristic curves naturally give rise to the function $k$ defined in \Cref{eq:definition_k}. 
Similar equations arise in problems related to random walks and combinatorics. In particular, in \cite{guerin2025elephantpolynomials,bercu2024sharp,bryc2009large}, expressions related to \cref{eq:equation_for_Laplace} discussed above are obtained by solving similar differential equations. Furthermore, a more general version encompassing several combinatorial subcases can be found in \cite{eulerian_recurrences}. Notably, in a more analytic setting, Dominici \cite{dominici2003nested} uses a similar equation to derive an expansion of the inverse of a function that can be written in a similar way to $k$. The recurring appearance of such equations across seemingly different contexts suggests that these models share a common combinatorial--analytic structure.

The singularity analysis method applied to $L_a(x,z)$ requires proving that it is analytic in a domain of form $\Delta(\epsilon,\theta)$ for some $\epsilon >0, \theta \in (0,\pi)$ (see \Cref{flajolet}). This is guaranteed by \Cref{analicity}.

\begin{lemma} \label{analicity}
   For every $a\in [-1,1)\setminus \{0\}$, the generating function $L_a(x,z)$ extends analytically to a domain $D$ that contains the half plane $\{ z \colon \Re(z) < k(\sinh(x))\sinh^\frac1a(x) \}$ and a circular sector $\Theta$ around $z_0 = k(\sinh(x))\sinh^\frac1a(x)$ defined as $$\Theta \coloneqq \left\{ z\in \C \colon \abs{z-z_0} < r, \abs{\arg( z - z_0 )} > \theta \right\},$$
   with $r>0, \theta \in (\pi,2\pi)$.
\end{lemma}

\Cref{analicity} is a key component of the analysis rather than a purely technical result. Indeed, the asymptotic behavior of the Laplace transform is entirely determined by the analytic structure of $L_a(x,z)$ in the complex $z$ variable, and in particular by the location and nature of its dominant singularity. The proof of the analytic continuation via a Schwarz--Christoffel mapping therefore provides the essential bridge between the probabilistic structure of the walk and its complex-analytic properties.

Moreover, once the singularity structure is identified, both the exponential growth rate of the coefficients and their polynomial corrections follow from singularity analysis techniques. In the case of the function studied here, this growth rate is related to the limiting cumulant generating function, while the subexponential terms capture the leading correction for the exact Laplace transform asymptotics.

\Cref{fig:plane_and_sector} shows the union of a half-plane and a circular sector with radius $r$ and angle $\theta \in (\pi,2\pi)$. As it is shown in the figure, a domain of the form $\Delta(\epsilon,\theta)$ can always be found in $D$ for every value of $x$ and $a$.

\begin{center}
\begin{figure}[!ht]
    \centering
    % PARAMETERS
\def\alphadeg{140} % angle of the sector in degrees
\def\R{2} % radius
\def\axeslen{3.5} % axes length
\def\h{3}
\def\eps{0.35}

% For a sector opening to the left, center it around 180°.
\pgfmathsetmacro{\halfalpha}{\alphadeg/2}
\pgfmathsetmacro{\angstart}{90}
\pgfmathsetmacro{\angend}{90+\halfalpha}

\begin{tikzpicture}[>=latex, thick, scale=0.6, yscale=1, xscale=-1]

% unified color fill
\definecolor{fillcolor}{RGB}{220,220,255}

% Right half-plane fill
\begin{scope}
\clip (0,-\axeslen) rectangle (2*\axeslen,\axeslen);
\fill[pattern={Lines[angle=45,distance={4pt},line width={0.01mm}]},pattern color=gray] (0,-\axeslen) rectangle (2*\axeslen,\axeslen);
\end{scope}

% Circular sector fill (opening to the left)
\fill[pattern={Lines[angle=45,distance={4pt},line width={0.01mm}]},pattern color=gray] (0,0) -- (\angstart:\R) arc(\angstart:\angend:\R) -- cycle;
\fill[pattern={Lines[angle=45,distance={4pt},line width={0.01mm}]},pattern color=gray] (0,0) -- (-\angstart:\R) arc(-\angstart:-\angend:\R) -- cycle;

% Angle measure
\draw[dashed, <-] (\angend:\R) arc(\angend:\angend+\alphadeg/2:\R);

% Annotations
\draw (-2, 0.6) node [anchor=north west][inner sep=0.75pt] {$\theta$};
\draw (\h-0.1, \axeslen+0.5) node [anchor=north west][inner sep=0.75pt] {$\Im(z)$};
\draw (-\axeslen+0.5,-1) node [anchor=south west][inner sep=0.75pt] {$\Re(z)$};

% Borders
\draw [line width=1.5] (0,\axeslen) -- (0,\R);
\draw [line width=1.5](0,-\axeslen) -- (0,-\R);
\draw [line width=1.5] (0,\R) arc(\angstart:\angend:\R);
\draw [line width=1.5] (0,-\R) arc(-\angstart:-\angend:\R);
\draw [line width=1.5] (0,0) -- ({\R*cos(\angend)},{\R*sin(\angend)});
\draw [line width=1.5] (0,0) -- ({\R*cos(-\angend)},{\R*sin(-\angend)});

% Axes only
\draw[-to] (2*\axeslen+0.3,0) -- (-\axeslen,0);
\draw[-to] (\h,-\axeslen) -- (\h,\axeslen+0.5);

% Pacman
\draw [dashed, line width=0.25] (\h,0) circle (\h);
\draw [line width=1.5] ([shift=(-6:{-\h-\eps})]\h,0) arc (174:-174:{\h+\eps});
\draw [line width=1.5] (-\eps, \eps) -- (0,0);
\draw [line width=1.5] (-\eps, -\eps) -- (0,0);
\draw (\h,0.1) -- (\h,-0.1);
\draw (0,-0.35) node [anchor=north east, fill= white][inner sep=0.75pt] {\small $z_0$};%{\small $\frac{k(\sinh(x))}{\sinh^{-1/a}(x)}$};

\end{tikzpicture}
    \caption{The shaded area is the union of a half-plane and a circular sector around $z_0$ with angle $2\pi-\theta$ where $\theta \in (0,\pi)$ and radius $r>0$. The function $L_a(x,z)$ is analytic inside the shaded area and has a singularity at $z_0$. For every $(x,a)$ it is possible to find $(\epsilon,\theta)\in \mathbb R_+\times (\pi,2\pi)$ such that $L_a(x,z)$ is analytic in a domain of the form $\Delta(\epsilon,\theta)$.}
    \label{fig:plane_and_sector}
\end{figure}
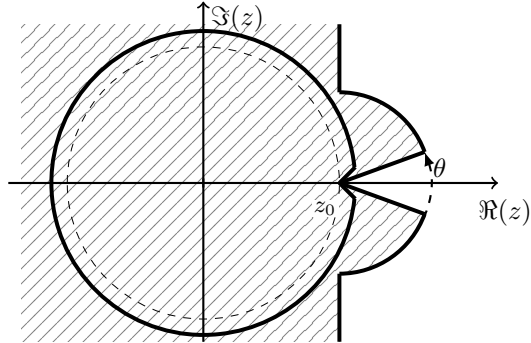
\end{center}

\Cref{analicity} allows one to find, for fixed $a,x$, the maximal value of $r_a(x)$ defined in \Cref{form_solution} as the inverse under multiplication of $\fgd$.

\begin{corollary}For every $a\in [-1,1]\setminus \{0\}$, define the maximal radius of convergence  in \Cref{form_solution} as \begin{equation*}
    \rho_a(x) \coloneqq \sup\left\{ r_a(x) \colon L_a(x,z) \text{ is analytic in } \abs{z} < r_a(x) \right\},
\end{equation*}where $r_a(x)$ is as defined in $\Cref{form_solution}$. Then $\rayon$ is given by $\rayon = 1/\fgd$.
\end{corollary}

\subsection*{Outline of the paper.}
\Cref{section:proofs} contains the proofs of the main results. \Cref{form_solution} is proved first, based on standard properties of the \gls{erw} and classical methods for first-order PDEs. 

The main technical component is the proof of \Cref{analicity}, which is divided into several cases: The regimes $a \in (-1/2,0)\cup(0,1)$ are treated in \Cref{subsection_a-120} using a complex-analytic differential equations approach, while the case $a\in (-1,-1/2]$ is handled in \Cref{subsection_a-1-12} via a Schwarz--Christoffel mapping. The proof of \Cref{result} is completed therein.

Finally, the proofs of \Cref{large_deviations} and \Cref{non_analytic} are provided in \Cref{last_subsection}. It also contains two additional results: the monotonicity and continuity at $a=0$ of the map $a \mapsto \phi(a;x)$, and a closed-form expression for $\rayon$ for specific values of $a$.

\section{Proof of the main results} \label{section:proofs}

\noindent\textbf{Remark.} Identity \eqref{eq:even} allows one to see that the values of $\L{n}$ for $x<0$ can be easily recovered from those for $x>0$ by evaluating at $1-q$ instead of $q$ whenever necessary. From now on, all of the computations will suppose $x>0$ unless stated otherwise.

\begin{proof}[Proof of \Cref{form_solution}]
A conditioning on $\mathcal F_n$ allows us to use \eqref{eq:recursion_probability} to find a differential-recursive equation for $\mathcal L_{n+1}(a;x)$ with $n\ge 1$:
\begin{align*}
    \L{n+1} &= \mathbb E \left[ e^{-xS_n} \left( \frac12 \left( 1 + a\frac{S_n}n \right)e^{-x} + \frac12\left( 1 - a\frac{S_n}n\right)e^{x} \right) \right]\\
    &= \L{n} \cosh(x) + \frac an \mathcal L_n'(a;x) \sinh(x).
\end{align*}

To make the sequence $\{\L{n}\}_{n\in\N}$ more convenient to work with, we make use of the reversible normalization for $x\in \R_+, n\ge 0, a \neq 0$ given by
\begin{equation*}
        D_n(x) \coloneqq \frac{ \L{n+1}}{\sinh^{\frac{a-n}{a}}(x)}
    \end{equation*}

The differential-recursive equation for $\L{n}, n\ge 1$ leads to 
    \begin{align*}
        D_{n}(x) = \frac{\mathcal L_{n+1} (a;x)}{\sinh^{\frac{a-n}{a}}(x)} &= \frac{\mathcal L_{n}(a;x)\cosh(x)}{\sinh^{\frac{a-n}{a}}(x)} + \frac an \frac{\mathcal L_{n}'(a;x)\sinh(x)}{\sinh^{\frac{a-n}{a}}(x)},\\
        &=\frac{\mathrm d}{\mathrm d x} \left[ n^{-1}\mathcal L_n(a;x) a\sinh^{\frac na}(x) \right] = \frac{\mathrm d}{\mathrm d x} \left[n^{-1}a\sinh^{\frac{a+1}{a}}(x) D_{n-1}(x) \right].
    \end{align*}
Then, if $$f(x) \coloneqq a\sinh^{\frac{a+1}{a}}(x),$$ the sequence $D_n(x)$ satisfies the relationship

    \begin{align}
        nD_{n}(x) = \frac{\mathrm d}{\mathrm d x} [f(x) D_{n-1}(x)]. \label{eq:recursion}
    \end{align}

Sequences $\{b_n(x)\}_{n\in \N}$ satisfying recursion \eqref{eq:recursion} and $b_0(x) \equiv 1$ are studied in  \cite{dominici2003nested} where they are called ``Nested derivatives of $f$". The sequence $\{D_n\}_{n\in\N}$ can be seen as a particular case as well as some other examples in the literature, particularly the elephant polynomials in \cite{guerin2025elephantpolynomials}.
    
For $x\in \R_+$, let $G_a(x,z)$ be defined as the generating function of the sequence $\left\{D_n(x)\right\}_{n\ge 0}$, i.e.
\begin{equation*}
    G_a(x,z) \coloneqq \sum_{n\ge0} D_n(x)z^n.
\end{equation*}
Then recursion \eqref{eq:recursion}, leads to differential equation for $G_a(x,z)$ in a small enough neighborhood of $0$,

\begin{align}
    \partial_x \left[ f(x) G_a(x,z) \right] &= \sum_{n\ge0}(n+1) D_{n+1}(x) z^n = \sum_{n\ge1} D_n(x) n z^{n-1} = \partial_z[G_a(x,z)]. \label{eq:diff_eq_G}
\end{align}

Equation \eqref{eq:diff_eq_G} is a first order transport equation which can be rewritten as
\begin{equation}
    \partial_z[G_a(x,z)] -  f(x)\partial_x[G_a(x,z)] = \partial_x[f(x)]G_a(x,z). \label{eq:transport}
\end{equation}
Let $x,z$ be parameterized by two real parameters $s,t$ such that $s$ satisfies
\begin{align}
    \partial_{s}[z(s,t)] &= 1, \label{eq:coordinate1}\\ \partial_s[x(s,t)] &= - f(x(s,t)), \label{eq:coordinate2}
\end{align}
and $t$ will be specified later. From now on, $G_a$ will denote $G_a(x(s,t), z(s,t))$ and $f$ will denote $f(x(s,t))$.
The condition for $s$ imply, by the chain rule and \cref{eq:transport}
\begin{align*}
    \partial_{s}[G_a] = \partial_z[G_a] - f\partial_x [G_a] = \partial_x[f]G_a(x,z). 
\end{align*}
\Cref{eq:coordinate2} allows one to obtain the simpler ODE:
\begin{align*}
    \frac{\partial_s G_a}{G_a} = \partial_x[f] &= -\frac{\partial_s[f]}{f},
    \intertext{so that,}
    \partial_s \left[ \log ( G_a f ) \right] &= 0.
\end{align*}
With this parametrization, $G_a f$ is constant with respect to $s$, but it depends on $t$. Solving \cref{eq:coordinate1,eq:coordinate2} yields
\begin{align*}
    z(s,t) &= s + c_1(t)\\
    \partial_s \left[-\int^{x(s,t)} \frac{\d u}{f(u)} \right] &= -\frac{\partial_s x(s,t)}{f(x(s,t))} = 1,
    \intertext{so that,}
    -\int^x \frac{\d u}{f(u)} &= s + c_2(t).
\end{align*}
It follows that
\begin{align*}
    z + \int^x \frac{\d u}{f(u)} = c_1(t) - c_2(t)
\end{align*}
does not depend on $s$, but only on $t$. In particular $t$ can be choosen to be $t = z + \int^x \frac{\d u}{f(u)}$. This remains a valid change of coordinates since $f(x)$ is positive for $x>0$. Therefore, this means that for some function $g$
\begin{align*}
    f(x)G_a(x,z) &= g(t) = g\left( z + h(x) \right),
    \intertext{where $h$ is the invertible function}
    h(x) &\coloneqq \frac1a\int^x \frac{du}{\sinh^{\frac{a+1}{a}}(u)} = \frac1a \int^x \frac{\d u}{f(u)}.
\end{align*}
To find such $g$, fix the initial condition $G_a(x,0) = D_0(x)$ and
\begin{align*}
    G_a(x,z) &= \frac{\sinh^{\frac{a+1}{a}}( h^{-1}( z + h(x) ) D_0\left( h^{-1}( z + h(x) ) \right)  }{ \sinh^{\frac{a+1}{a}}(x )}.
\end{align*}
Substituting $D_0=\mathcal L_1(a;x)/\sinh(x)$, yields an explicit expression for $G_a(x,z)$

 \begin{equation*}
     G_a(x,z) = \frac{ [\mathcal L_1 \times \sinh^{\frac1a}] \circ h^{-1}(z+h(x)) }{ \sinh^{\frac{a+1}{a}}(x)}.
 \end{equation*}

From this point, suppose $a<0$. Since $1/\sinh^{\frac{a+1}a}(u)$ is integrable in $[0,c]$ for any $c<\infty$, the integration limits in $h$ can be choosen to be $0$ and $x$. The case $a>0$ is analogous with the choice $x$ and $\infty$. Let $k$ be defined as in \cref{eq:definition_k}, $k(t) = -\frac1a\int_0^t s^{-\frac{a+1}{a}}(1+s^2)^{-\frac12} \d s$. Evaluating $h$ at $\arcsinh(x)$ shows that

    \begin{equation*}
        h(\arcsinh(x)) = \frac1{a}\int_0^{\arcsinh(x)} \sinh^{-\frac{a+1}{a}}(s) \d s + C = \frac1{a}\int_0^{x} \frac{s^{-\frac{a+1}{a}}}{\sqrt{s^2 +1 }}\d s + C = C - k(x).
    \end{equation*}

So $h(x) = C - k(\sinh(x))$ and $h^{-1}(x) = \arcsinh( k^{-1}( C - x) )$. This allows one to write
 \begin{align*}
 G_a(x,z) &= \frac{ [\mathcal L_1 \times \sinh^{\frac1a}] \circ \arcsinh[k^{-1}(C - z - h(x))] }{           \sinh^{\frac{a+1}{a}}(x)} \\
 &= \frac{ \mathcal L_1[\arcsinh(k^{-1}(C-z-h(x)))]  [k^{-1}(C-z-h(x))]^\frac1a }{ \sinh^{\frac{a+1}{a}}(x)}.
  \end{align*}
        Then, since $\mathcal L_1(x) = qe^{-x} + (1-q)e^{x} = -2q\sinh(x) + e^{x}$ and $\arcsinh(x) = \ln (x + \sqrt{1 + x^2})$, 
 \begin{equation}
 G_a(x,z) = \sinh^{- \frac{a+1}{a}}(x) A\left( k^{-1}( k(\sinh(x)) - z + C ) \right), \label{eq:expression_for_G}
 \end{equation}
with
\begin{equation*}
  A(t)= \left( -2qt + t +\sqrt{1+t^2} \right)t^\frac1a. \notag
 \end{equation*} 
    Notice that
    \begin{align*}
        G_a(x,z) &= \sum_{n\ge 0} \frac{\mathcal L_{n+1}(a;x)}{\sinh^{\frac{a-n}{a}}(x)} z^n = \sum_{n\ge 0} \frac{\mathcal L_{n+1}(a;x)}{\sinh(x)} \left(z\sinh^{\frac1a}(x)\right)^n = \frac{L_a(x,z\sinh^{\frac1a}(x))}{\sinh(x)}.
    \end{align*}
    So, the substitution $w = z\sinh^\frac1a(x)$ in \eqref{eq:expression_for_G} leads to
    \begin{equation*}
        L_a(x,w) = \sinh^{-\frac{1}a} (x)A\left[ k^{-1}\left( k\left(\sinh(x)\right) - w\sinh^{-\frac1a}(x) + C \right) \right].
    \end{equation*} 
    Finally, since $L_a(x,0) = \L{1}$, we choose $C=0$.
    \end{proof}

The form of $L_{-1}(x,z)$ allows one to conclude \Cref{analicity} and \Cref{result} for the specific case $a=-1$, since in this case $k(t) = \arcsinh(t)$ so $k^{-1}(t) = \sinh(t)$, and 
\begin{equation}
    L_{-1}(x,z) = \sinh(x)\left( 1 - 2q + \mathrm{cotanh}( x + z\sinh(x)) \right), \label{eq:generatrice-1}
\end{equation}
which is holomorphic on $z$ in $\C\setminus \mathbb Z \pi i$ with simple poles at the integer multiples of $\pi$. Its asymptotic expansion around 0 together with \Cref{flajolet} introduced below implies 
\begin{equation}
    \mathcal L_n(-1;x) \sim \left( \frac{\sinh(x)}{x} \right)^n, \qquad n \to \infty.\label{eq:minus_one}
\end{equation}
This coincides with \cref{eq:asymptotic_laplace}. Interestingly, the asymptotic Laplace transform coincides with the Laplace transform of the sum of $n$ i.i.d. random variables with uniform distribution in $[-1,1]$. Notice also from \eqref{eq:generatrice-1} that for $a=-1$, the parameter $q$ affects only the first coefficient in the expansion of $L_{-a(x,z)}$, i.e. only the Laplace transform of the first step. This is due to the fact that the second step of the walk is always in the opposite direction of the first one, and since the next step chooses randomly between all of the past ones, further steps are no longer affected by the initial one.

For the rest of the values of $a$, the proof of \Cref{analicity} will be divided in two cases: the first one is when $a\in (-1/2,0)$ or $a\in(0,1)$ which will be shown to be equivalent; the second one is when $a\in (-1,-1/2)$ which uses a different set of tools.

Before proceeding to the proofs, it is convenient to state more precisely the result that will be used to estimate the coefficients of $L_a(x,z)$.

\begin{proposition}(Flajolet - Odlyzko \cite{flajolet1990singularity}) \label{flajolet} 
Assume that, with the sole exception of the singularity $z=1$, $F(z)$ is analytic in the domain $\Delta(\epsilon,\theta) \coloneqq \{z : \abs{z} \le 1 + \epsilon,  \abs{\mathrm{Arg}(z-1)} \ge \theta \}$ where $\epsilon > 0$ and $\pi/2 <\theta < \pi$. Assume further that as $z$ tends to 1 in $\Delta$,
     \begin{equation*}
         F(z) \sim G(z),
     \end{equation*}
     Then, as $n\to\infty$, the nth Taylor coefficient of $F(z)$ satisfies 
     \begin{equation}
         [z^n]F \sim [z^n]G
     \end{equation}
\end{proposition}
More details about the proposition, and the method in general, can be found in \cite{flajolet1990singularity}.

To show how the analyticity of $A(k^{-1}(z))$ depends on the functions involved, let $a \in (-1,1)\setminus \{0\},x\in \R_+$ be fixed, from \Cref{form_solution}.  The function $A\colon \C\cup\{\infty\} \to \C\cup\{\infty\}$ has branch points at $0,\infty, i$ and $-i$. Branch cuts can be chosen so that $A$ is analytic in $\C \setminus \left( \R_-\cup C(-i,i)  \right)$, where $C(-i,i)$ is an analytic arc joining $-i$ and $i$ to $0$. In particular, $C(-i,i)$ can be chosen to lie outside a sector around $0$, $\Theta_0(r,\theta)=\left\{ z\in \C \colon \abs{z} < r, \abs{\arg(z)} > \theta \right\}$ with $r>0, \theta \in(\pi,2\pi)$. Then $A$ is analytic in the domain $D$ defined in \Cref{analicity}, and $A(k^{-1}(z))$ will be analytic in $D$ if $D \subset k(D)$, as seen in the diagram:

\begin{align*}
    D \subset k(D) \xrightarrow[\phantom{ooooooooooo}]{k^{-1}} D \xrightarrow[\phantom{ooooooooooo}]{A} A(D).
\end{align*}

This allows one to see that the proof of \Cref{analicity} consists essentially of investigating the analyticity of $k^{-1}$, which is the purpose of \Cref{plane_inversion,local_inversion,sc,sc-mapping}.

The inversion of $k$ for $a\in(0,1)$ and $a\in(-1/2,0)$ on the right half-plane is seen to be equivalent by a composition with the Möbius function $f(z) = 1/z$. Take $a_+ \in (0,1)$, and write
\begin{equation*}
    k_{a_+}(z) = \frac1{a_+} \int_z^\infty \frac{\d s}{s^{1+\frac1{a_+}}\sqrt{s^2+1}},
\end{equation*}
then, by the change of variable $t = 1/s$
\begin{align*}
    k_{a_+}\left( \frac1z\right) &=  \frac1{a_+}\int_0^z \frac{\d t}{t^{1 + \frac{-1-a_+}{a_+}}\sqrt{1+t^2}} = \frac1{a_+}\int_0^z \frac{\d t}{t^{1 + \frac1{a_-}}\sqrt{1+t^2}} = -\frac{a_-}{a_+}k_{a_-}(z), \label{eq:positive_negative}
\end{align*}
where $a_- = \frac{-a}{a+1} \in (-1/2,0)$, and $k_{a_-}(z) = -\frac1{a_-} \int_0^z \frac{t^{-1-\frac1{a_-}}\d t}{\sqrt{t^2+1}}$.
Since $\frac1z$ is holomorphic from the right half-plane to itself, the problem of inverting $k$ on $\rightplane$ is equivalent for $a \in (-1/2,0)$ and $a\in(0,1)$.

\subsection{Proof of \Cref{analicity} for the case in (-1/2,1)} \label{subsection_a-120}

\begin{proposition} \label{plane_inversion}
    For $a \in (-\frac12,0)\cup(0,1)$, the function $k: \C \to \rightplane$ given by \Cref{eq:definition_k} has an inverse that is analytic in $\rightplane$. 
\end{proposition}

\begin{proof} Since $k_{a_+}\left( \frac1z\right) = -\frac{a_-}{a_+}k_{a_-}(z)$, it is possible to take $a\in(-1/2,0)$ without loss of generality.

By the chain rule:
    \begin{align}
        (k^{-1})'(k(z)) &= \frac{1}{k'(z)} = -a z^{1+\frac1a}\sqrt{z^2+1}
         \intertext{so, $k^{-1}$ must satisfy the differential equation}
         (k^{-1})'(z)    &= -a(k^{-1}(z))^{1+\frac1a}\sqrt{(k^{-1}(z))^2 + 1} \eqqcolon g(k^{-1}(z)), \label{eq:diffeq}
         \intertext{with}
         g(u) &= -a u^{1+\frac1a}\sqrt{ u^2 + 1 } \label{eq:def_g}
    \end{align}
    Notice that $g$ is holomorphic in the right half-plane, so the local existence and uniqueness to \eqref{eq:diffeq} is guaranteed in $\rightplane$ (see Theorem 4.1 in \cite{teschl_complex_diffeqs}). Moreover, since $k:\R_+ \to \R_+$ is strictly increasing,  \eqref{eq:diffeq} admits an unique solution in a neighborhood of the real positive axis $\{x\in  \R \colon x > 0\}$. Fix $x > 0$, $I \subset \R$ an interval around 0, and let $f_x: I \to \C$ be given by $f_x(y) = k^{-1} (x + iy)$ for. Then $f_x$ satisfies the differential equation
    \begin{equation}
        f_x'(y) = i g(f_x(y)), \label{eq:real_diffeq}
    \end{equation}
    with $g$ as defined in \eqref{eq:def_g}. 
    
    A first-order Taylor expansion of $ig( u + iv )$ around $u > 0$ gives
    \begin{align*}
        ig(u+iv) &= ig(u) - v g'(u) + o(\abs{v}).
    \end{align*}
    Since $g'(u) < 0$ when $u>0$, we have that for $v$ close to 0, $\Im(ig(u+iv))>0$ and $\sgn(\Re(ig(u+iv))) = \sgn(v)$. For every fixed $u>0$, define $v_i(x), v_r(x)$ as
    \begin{align*}
        v_i(u) &\coloneqq \inf \left\{ v > 0 \vcentcolon \Im( ig(u + iv) ) \le 0 \right\} \in \R_+ \cup \{ + \infty\},\\
        v_r(u) &\coloneqq \inf\left\{  v > 0 \vcentcolon \Re( ig(u + iv) ) \le 0 \right\} \in \R_+ \cup \{ + \infty\}.
    \end{align*}
    
    For $n \in \mathbb N$, let $\mathcal D_n$ be defined as 
    
    \begin{equation*}
        \mathcal D_n \coloneqq \left\{ x + i y \vcentcolon x > \frac1n, \abs{y} < v_i(x) \right\},
    \end{equation*}
    and $\mathcal D \coloneqq \bigcup_j \mathcal D_j$, so $\mathcal D = \left\{ x + i y \vcentcolon \abs{y} < v_i(x) \right\}$. Since for every $z\in \mathcal D$, it exists $\mathcal D_n$ such that $z\in \mathcal D_n$, the result of existence and uniqueness of a solution to \eqref{eq:real_diffeq} until the exit time can be extended to $\mathcal D$. Proving that an extension of $f_x(y)$ from the real positive axis along the vector field $ig(x + iy)$ never reaches the boundary $\partial \mathcal D$, then implies the desired global existence and uniqueness. 
    
    Since $ig(z)$ is holomorphic on every $\mathcal D_n$, its derivative is locally bounded, therefore $ig(z)$ is locally Lipschitz and the Picard-Lindelöf Theorem yields the existence and uniqueness of a solution to \eqref{eq:real_diffeq} for every initial condition $z_0 \in \mathcal D_n$ that can be continued  until the exit time of the set $\mathcal D_n$. Furthermore, if there exists $x_0$ such that $v_i(x)=v_r(x)=\infty$  for all $x>x_0$, then $f_x(y)$ admits an analytic continuation to all $y\in\mathbb R_+$. Indeed, for $x>x_0$, the inequalities
$$
\Im\left(i g(x+iy) \right)>0,
\qquad
\Re\left(i g(x+iy) \right)>0
$$
show that the vector field $i g$ points strictly into $\mathcal D_n$ along the relevant portions of the boundary. Consequently, every integral curve starting at a point $x+iy$ with $x>x_0$ remains in $\mathcal D_n$ for all forward times and therefore cannot intersect $\partial\mathcal D_n$. It follows that the analytic continuation can be extended for all $y\ge 0$.  \Cref{fig:vector_fields} shows the vector field of $ig(x+iy)$ and the curve $v_i(x)$ for different values of $a$.

\begin{figure}[!ht]
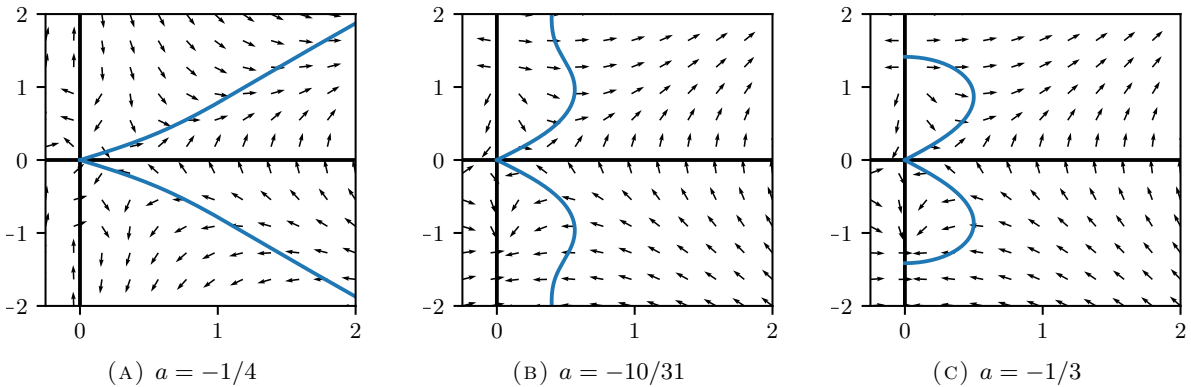

    \centering
        \def\mathdefault#1{#1}\everymath=\expandafter{\the\everymath\displaystyle}
    \begin{subfigure}[t]{0.3\textwidth}
        \centering
        \scalebox{0.95}{\input{figures/vector_field_1_4.pgf}} 
        \caption{$a=-1/4$}\label{subfig:vector_field_14}
    \end{subfigure} \hfill
    \begin{subfigure}[t]{0.3\textwidth}
        \centering
        \scalebox{0.95}{\input{figures/vector_field_10_31.pgf}}
        \caption{$a=-10/31$}\label{subfig:vector_field_1031}
    \end{subfigure}\hfill
    \begin{subfigure}[t]{0.3\textwidth}
        \centering
        \scalebox{0.95}{\input{figures/vector_field_1_3.pgf}}
        \caption{$a=-1/3$}\label{subfig:vector_field_13}
    \end{subfigure}
    \caption{Vector fields of $ig(x+iy)$ and the curve $v_i(x)$ for different values of $a$. In \Cref{subfig:vector_field_13} and \Cref{subfig:vector_field_1031} the value of $v_i(x)$ is eventually infinite.}
    \label{fig:vector_fields}
\end{figure}

    By continuity of $ig(z)$ in $\rightplane$, finding $v_i(x)$ is equivalent to finding the smallest $y$ such that the imaginary part of $ig(x+iy)$ vanishes, and similarly for $v_r(x)$. This happens if either $\Re(ig(z)) = 0$ or $\Im(ig(z))=0$, then $-g^2(z) \in \R$, with $-g^2(z) < 0$ if $\Re(ig(z)) = 0$ and $-g^2(z)>0$ if $\Im(ig(z)) = 0$. The trivial points such that $-g^2(z) = 0$ are excluded since they're not in $\rightplane$.
    
    If $z = re^{it}$, the condition $-g^2(z) \in \R$ is equivalent to
    \begin{align}
        r^2\sin( 2t( 2+\frac1a ) ) + \sin( 2t(1 + \frac1a) ) &=0, \label{eq:condtion_both}
        \intertext{while the condition $-g^2(z) < 0$ and $-g^2(z)>0$ are equivalent, respectively, to}
        r^2\cos( 2t( 2+\frac1a ) ) + \cos( 2t(1 + \frac1a) ) &<0, \label{eq:condition_real_part}\\
        r^2\cos( 2t( 2+\frac1a ) ) + \cos( 2t(1 + \frac1a) ) &>0. \label{eq:condition_imag_part}
    \end{align}
    Condition \eqref{eq:condtion_both} admits a degenerate solution when $\sin(2t(2+1/a)) = \sin(2t(1+1/a)) = 0$, and this happens if and only if there exist $n,m\in \mathbb Z$ such that $2t(2+1/a) = n\pi, 2t(1+1/a) = m\pi$. From these conditions it follows
     \begin{align*}
         t &= n\left(\frac{a}{2a+1}  \right)\frac\pi2,\\
         m\pi &= 2 n\left(\frac{a}{2a+1}  \right)\frac\pi2 (1+1/a).
     \end{align*}
     Then $\frac mn = \left( \frac{a+1}{2a+1} \right)$. Since $\frac{a+1}{2a+1} > 1$, it is possible to take $m,n\in \mathbb N$, $m > n$, and writing $a = \frac{n-m}{2m-n}$ yields that condition \cref{eq:condtion_both} only holds degenerately for particular values of $a$. When there is no degenerate solution, it is possible to write $$ r^2=\frac{\sin\left(-2t(1+1/a)\right)}{\sin\left(2t(2+1/a)\right)}.$$ A real value of $r$ exists if and only if  $$\frac{\sin\left(-2t(1+1/a)\right)}{\sin\left(2t(2+1/a)\right)}>0.$$ 
     Equivalently, there exist $m,n\in\mathbb Z$ such that $t$ belongs to one of the sets
     \begin{align}
         \left( - \frac\pi2(2n+1)\frac{a+1}{a}, -\frac\pi2(2n+2)\frac{a+1}{a} \right) \cap \left( \frac\pi2 (2m+2)\frac{a}{2a+1}, \frac\pi2 (2m+1)\frac{a}{2a+1} \right), \\
          \left( - \frac\pi2(2n)\frac{a+1}{a}, -\frac\pi2(2n+1)\frac{a+1}{a} \right) \cap \left( \frac\pi2(2m+1)\frac{a}{2a+1}, \frac\pi2 (2m)\frac{a}{2a+1} \right).
     \end{align}
     For $a\in (-1/2,0)$, there is no $t \in \left( \frac\pi2(\frac a{a+1}), \frac\pi2\left( \frac{-a}{a+1}\right)\right)$ that satisfies any of the previous conditions. This means that inside the cone $\{\rho e^{i\theta} \colon \rho > 0, \theta \in \left( \frac\pi2(\frac a{a+1}), \frac\pi2\left( \frac{-a}{a+1}\right)\right)\}$, the real and imaginary parts of the vector field $ig(z)$ are always positive. 
     
     When \eqref{eq:condtion_both} is satisfied, $r^2$ can be written as $r^2 = \sin(-2t(1+1/a))/\sin(2t(2+1/a))$, and with this the left-hand-side of both \eqref{eq:condition_imag_part} and \eqref{eq:condition_real_part} is equal to $\sin(2t)/\sin(2t(2+1/a))$. Condition \eqref{eq:condition_imag_part} is satisfied whenever there exist $m,n\in \mathbb Z$ such that
     \begin{align}
         t \in \left( n\pi, n\pi + \frac\pi2 \right) \cap \left( \frac{2m+1}{2}\pi \left( \frac{a}{2a+1} \right), m\pi \left( \frac{a}{2a+1} \right) \right),
         \intertext{or}
         t \in \left( n\pi + \frac\pi2, (n+1)\pi \right) \cap \left( m\pi \left( \frac{a}{2a+1} \right), \frac{2m-1}{2}\pi \left( \frac{a}{2a+1} \right) \right).
     \end{align}
     Condition \eqref{eq:condition_real_part} is satisfied whenever
     \begin{align}
          t \in \left( n\pi + \frac\pi2, (n+1)\pi \right) \cap \left( \frac{2m+1}{2}\pi \left( \frac{a}{2a+1} \right), m\pi \left( \frac{a}{2a+1} \right) \right),
          \intertext{or}
         t \in \left( n\pi, n\pi + \frac\pi2 \right) \cap \left( m\pi \left( \frac{a}{2a+1} \right), \frac{2m-1}{2}\pi \left( \frac{a}{2a+1} \right) \right),
     \end{align}
     Define $t_{\inf} \coloneqq \min\{ \frac\pi2, \frac{-a}{2a+1}\frac\pi2 \}$ and $t^{\sup} \coloneqq \min\left\{ \frac\pi2, \frac{-a}{a+1}\frac{\pi}{2} \right\}$. Former conditions and the restriction $t \in (0,\pi/2)$, imply that $x + i v_r(x)$ must satisfy:
    \begin{align}
        x + i v_r(x) &= \sqrt{ \frac{\sin\left(-2t \left( 1 + \frac1a \right) \right)}{\sin\left(2t \left( 2 + \frac1a \right) \right)} } \left( \cos t + i \sin t \right), \qquad t \in ( t_{\inf}, t^{\sup} ), \label{eq:parametrized_real}
        \intertext{and $x + iv_i(x)$}
        x + i v_i(x) &= \sqrt{ \frac{\sin\left(-2t \left( 1 + \frac1a \right) \right)}{\sin\left(2t \left( 2 + \frac1a \right) \right)} } \left( \cos t + i \sin t \right), \qquad t \in \left( \frac\pi2\left( \frac{-a}{a+1} \right), t_{\inf} \right). \label{eq:parametrized_imag}
    \end{align}
    
    If it happens that $t_{\inf} = t^{\sup}$, then $v_r(x) = \infty$ for all $x>0$. Also, the fact that $v_r(x)$ and $v_i(x)$ are divided by the line $\arg(z) = t^{\sup}$, implies that $v_r(x) \ge v_i(x)$, and so $\Re( ig(z) ) > 0$ inside of $\mathcal D$. Since $\frac{\pi}{2} \ge t^{\sup} \ge t_{\inf}$, then $\rightplane \subset \mathcal D$ if $t_{\inf} = \frac{\pi}{2}$.

    The function $r^2(t)$ has the following limits
    \begin{align*}
        \lim_{t\to \frac\pi2\left( \frac{-a}{a+1} \right)} r^2 &= 0,
         \qquad \text{and} \qquad
        \lim_{t \to t_{\inf}} r^2(t) = +\infty.
    \end{align*}
    However, since $\cos(\pi/2) = 0$, if $t_{\inf} = \pi/2$ it happens that the parameterized curve defined in \eqref{eq:parametrized_imag} touches the point $i$, and by continuity, it is bounded. This implies that there exists $x_{\max}\ge 0$ such that $v_i(x) < \infty$ for $0 < x \le x_{\max}$, and $v_i(x) = +\infty$ for $x > x_{\max}$. Then, the solution to \eqref{eq:real_diffeq} can be continued from $x>x_{\max}$ for every $y \in \R$. Since $r(t)$ is increasing between $\frac\pi2\left( \frac{-a}{a+1}\right)$, and $x_{\max}$, the horizontal tangent vectors to the vector field $ig(z)$ at $x+iv_i(x)$ point inside $\mathcal D$. Then, if $v_i(x) <0$, the solution to \eqref{eq:real_diffeq} can also be continued from $x$ for all $y\in \R$ (see \Cref{fig:vector_fields}).
    
    As a conclusion, for all $x>0$, the solution of \eqref{eq:real_diffeq} can be continued for all $y \in \R$. Then the global existence and uniqueness of a solution to \eqref{eq:real_diffeq} follows for every starting condition $x+iy$ with $x>0$. The local existence and uniqueness is guaranteed for \eqref{eq:eqdiff} with any starting condition on $\rightplane$, and since every solution to \eqref{eq:eqdiff} should satisfy \eqref{eq:real_diffeq}, the global uniqueness of the solution to \eqref{eq:real_diffeq} on $\rightplane$ implies that the solution to \eqref{eq:diffeq} should also be globally defined.
\end{proof}

\begin{remark}
When $a\in (-1,-1/2)$, the arguments used in the preceding proof do not work exactly the same way, particularly the behavior of $ig(z)$ around the positive real axis. Although some changes could be made to adapt the proof, the case $a\in (-1,-1/2)$ is easily seen to be a particular case of a special function whose analyticity domain is known.
\end{remark}

\begin{proposition} \label{local_inversion}     For $a \in (-\frac12,0)$, the function $k(z) = -\frac1{a}\int_0^z \frac{\d t}{t^{1+\frac1a}\sqrt{t^2 + 1}} \d t$ has an inverse that is analytic in a circular sector $\Theta(r_0,\theta_0)$ around 0 with $\theta_0 \ge \pi$.
\end{proposition}

\begin{proof}
Let $\abs{z}<1$, and denote with $b^{(n)}$ the ascending factorial $b^{(n)} =(b)(b+1)\cdots(b+n-1)$. By the binomial Theorem
\begin{align}
    k(z) = -\frac1a\int_0^z \frac{\d t}{t^{1+\frac1a}\sqrt{t^2 + 1}} \d t &= -\frac1a\int_0^z \frac{ \sum_{n \ge 0} \binom{-\frac12}{n} t^{2n} }{t^{1+\frac1a}} \d t = - \frac1a \sum_{n\ge0} \frac{(-1)^n (1/2)^{(n)}}{(2n-1/a)n!} z^{2n-\frac1a} \notag \\
    &= z^{-\frac1a}\sum_{n\ge0} \frac{(-1)^n(1/2)^{(n)}\left( -\frac{1}{2a}\right)^{(n)}}{\left( \frac{2a-1}{2a} \right)^{(n)}}(-z^2)^{n} = z^{-\frac1a} {}_2 F_1\left[ \frac12,-\frac1{2a} ; 1 - \frac1{2a}; -z^2 \right], \label{eq:hypergeometric}
\end{align}
where  the order of sum and integration can be switched since the integral converges for $\abs{z}<1$.
Denote \begin{align*}
    g(z) = z\left( {}_2 F_1\left[ \frac12,-\frac1{2a} ; 1 - \frac1{2a}; -z^2 \right] \right)^{-a}.
\end{align*}
Since $z^{\frac1a}k(z)$ is holomorphic in the unitary disc and $\left.z^{\frac1a}k(z)\right|_{z=0} = 1$, then $g(z)$ is holomorphic in a neighborhood of $0$. Furthermore
\begin{align*}
    \left.\partial_z g(z)\right|_{z=0} &= \left( {}_2 F_1\left[ \frac12,-\frac1{2a} ; 1 - \frac1{2a}; 0 \right] \right)^{-a} = 1.
\end{align*}
So $g(z)$ admits an holomorphic inverse in its range which is a domain round zero, so there exists a disc around 0, say of radius $\delta$, $\mathbb D_\delta$ where $g^{-1}(z)$ is holomorphic. Since $k(z) = \left(g(z)\right)^{-\frac1a}$ it follows that $k^{-1}(z) = g^{-1}( z^{-a} )$. Fixing a branch of logarithm, $z^{-a}$ is holomorphic in a slit disc of radius $\epsilon$, $\mathbb D_\epsilon\setminus  [0,\epsilon)$. Thus, $k^{-1}$ is holomorphic in the intersection of the domains: $\mathbb D_{\alpha}\setminus [0,\infty)$, where $\alpha = \min\left\{ \epsilon, \delta \right\}$. Choosing any $\theta_0 \in (\pi,2\pi)$ finishes the proof.
\end{proof}

\begin{corollary}
For $a \in (0,1)$, the function $k(z) = \frac1a\int_z^\infty \frac{\d s}{ s^{1+\frac1a}\sqrt{s^2 + 1} }$ has an inverse that is analytic in $\Theta(r_0, \theta_0)$ a circular sector around 0 with angle $\theta_0$ and radius $r_0$.
\end{corollary}

\begin{proof}
Let $b = -a/(a+1) \in (-1/2,0)$. By \cref{eq:positive_negative}, 
\begin{equation*}
     k\left( \frac1z \right) = \frac1a \int_0^z \frac{\d s}{s^{1+\frac1b}\sqrt{s^2 + 1} }.
\end{equation*}
Then \Cref{local_inversion} guarantees that $k(1/z)$ has an inverse, say $g$ that is analytic in a circular sector around $0$. Since $g$ is injective near 0, and $g(0)=0$, after taking a sector with a  small enough radius and removing the origin, $g$ does not vanish there. Writing $k^{-1}(z) = 1/g(z)$ yields that $k^{-1}(z)$ is holomorphic in this circular sector.
\end{proof}

\subsection{Proof of \Cref{analicity} for the case in (-1,-1/2)} \label{subsection_a-1-12}

When $a\in(-1,-1/2)$, $k$ happens to be a conformal application between $\rightplane$ and a polygon on $\C$.  The Riemann Mapping Theorem ensures the existence of a conformal transformation between any two simply connected domains of the complex plane $\C$ which are different to the plane itself. The Schwarz-Christoffel mappings are a class of such conformal applications between a half-plane and a simply connected area bounded by right sides. 

More precisely, let $\alpha_1,\dots,\alpha_n$ be positive real numbers and $w_1,\dots,w_n$ be points in the complex plane. The polygon $P$ is defined as the area surrounded by the lines joining the points $w_j$ with internal angles $\pi\alpha_j$. The points $w_j$ are called vertices. If $F$ maps the upper half-plane conformally into to $P$, the preimages of the $w_j$ by $F$ are the real numbers $z_1,\dots,z_n$, and we call them \textit{prevertices}. For particular examples of polygons see \Cref{fig:polygons}. The general Schwarz-Christoffel formula for a conformal mapping $F$ between the upper half-plane and the polygon $P$, $F:\C^+\to P$ is

\begin{equation*}
        F(z) = A + C \int^z \prod_{j=1}^{n-1}(s-z_j)^{\alpha_j-1} \d s,
\end{equation*}
for some complex constants $A,C$ and some $z_1<z_2<\cdots <z_n$ with $w_j = f(z_j)$ for $j=1,\dots,n-1$. An extensive review of the Schwarz-Christoffel maps can be found in \cite{driscoll2002schwarz}.

\begin{proposition}\label{sc}
    For $a\in(-1,-1/2]$, let $r = \abs{k(i)}$ and $(w_1,w_2) = (re^{-i\frac{\pi}{2a}}, re^{i\frac{\pi}{2a}})$. Let $\Gamma$ be the line formed by the segments $ S_1 = (w_1,0)$, $S_2 = (0,w_2)$ and the rays $R_1,R_2$, where $R_i$ has an end in $w_i$, is perpendicular to $S_i$ and extends to infinity in the direction of the positive real axis. Define $P$ as the region of the complex plane delimited by $\Gamma$ containing the positive real axis (see \Cref{fig:polygons}, also, for examples of this polygon depending on $a$, see \Cref{interactive_polygon}). There exist $A,C$ complex constants, $C \neq 0$, such that $\phi: \C^+ \to P$ defined as  $$ \phi(z) = A + C\int_0^z s^{-1-\frac1a}(s-1)^{-\frac12}(s+1)^{-\frac12} \d s$$ is a conformal application between the upper half-plane $\C^+$ and $P$ with $\phi(-1) = w_1, \phi(0) = 0, \phi(1)=w_2$.
\end{proposition}

For the sake of completeness, the statement will be proven using the standard proof of the Schawrz-Christoffel mapping for general polygons found in \cite{henrici1993applied}.

\begin{proof}
    The existence of a conformal map from $\C^+$ to $P$ is guaranteed by the Riemann mapping theorem. Let $F$ be such a function.  By the Carathéodory-Osgood Theorem, it extends to an homeomorphism from $\C^+\cup \R$ to $\bar{P}$ (see, for example \cite[page 238]{stein2010complex}). With this continuous extension, the Schwarz reflection principle can be applied across the segments $s_1 = (-\infty,-1), s_2 = (-1,0), s_3 = (0,1)$ and $s_4 = (1,\infty)$ to extend $F$ analytically to the bi-infinite strips $\{\Re(z) < -1\}, \{-1<\Re(z) < 0\}, \{0<\Re(z) <1\}, \{1<\Re(z)\}$. Then $F$ is locally injective in each of the bi-infinite strips, and in particular its derivative exists for every $z_0\in \R \setminus \{-1,0,1\}$ and it is never 0. Let $\tau(z_0)$ be the tangent vector to $\Gamma$ at the point $F(z_0)\in \Gamma$ so that $$\arg\tau(z_0) = \arg(F'(z_0)).$$ This means $\arg(F'(z))$ is constant in each of the segments $s_j$, and at the prevertices $-1,0,1$ it has a jump depending on the angle of the vertex. An examination of the vectors tangent to the contour $\Gamma\setminus\{-1,0,1\}$, in counterclockwise order yields
    \begin{align*}
        \arg(F'(-1^+)) - \arg(F'(-1^-)) &= \frac\pi2 = \pi\left(1-\frac12\right), \\
        \arg(F'(0^+)) - \arg(F'(0^-)) &= \pi+\frac\pi a = \pi\left(1+\frac 1a\right), \\
        \arg(F'(1^+)) - \arg(F'(1^-)) &= \frac\pi2=\pi\left(1-\frac12\right), 
    \end{align*}
    i.e., it corresponds to the external angles of $P$ at each vertex.
    
    Thus, around each prevertex $z_j$, $F'(z)$ behaves locally as $(z-z_j)^{\alpha_j-1}$, where $\alpha_j$ is the corresponding internal angle at the vertex $w_j$. In a neighborhood of $z_j$, $F'$ can be written as 
    
    \begin{equation}
        F'(z) = (z-z_j)^{\alpha_j-1} \psi(z), \label{eq:analytique_near_vertex}
    \end{equation}
    with $\psi$ an analytic function in this neighborhood. 
    
    An application of the Schwartz reflection principle along the real axis gives that the image by $F$ of the lower half plane $\C^-\coloneqq \{\Im(z) < 0\}$ is a reflection $P'$ of $P$ with respect one of its sides. Another application of the Schwarz reflection principle takes the upper half plane $\C^+$ to a new reflection $P''$ of $P'$ with respect to one of the sides of $P'$. Continuing this procedure, it is possible to see that the image of $\C^+$ or $\C^-$ by $F$ is always an affine transformation of $P$, so any branch of $F$ can be written as $A + CF$ for some constants $A,C$. Then, the function $$\frac{(A+CF(z))''}{(A+CF(z))'} = \frac{F''(z)}{F'(z)}$$ is univariate in the complex plane, except at three possible singular points, since it is independent of the constants $A,C$. Furthermore, by \eqref{eq:analytique_near_vertex} and the fact that $F$ is conformal, $F''/F'$ has a simple pole with residue $\alpha_j-1$ at the prevertex $z_j$, and the function
    
    \begin{equation*}
        \frac{F''(z)}{F'(z)} - \sum_{j=1}^3 \frac{\alpha_j-1}{z-z_j}
    \end{equation*}
    is entire. Let us prove that it is also bounded.
    
   Take $z \in E_M \coloneqq \C^+\cap \{w \in \C : \abs{w} \ge M\}$. Since the angle between $S_1, S_2$ at infinity is $-\pi(1+\frac1a)$, it is possible to conclude that $F(z) \sim c_0z^{-1-\frac1a}$ for $z\in E_M$ and a constant $c_0\in\C$ when $M\to\infty$ (see \cite[Theorem 5.12d]{henrici1993applied}.) 
   
    With this, define the function $H$ as $$H(z) \coloneqq \left( F\left(\frac1z\right) \right)^{\frac a{a+1}}.$$ When $z\to 0$, $F(1/z)$ grows to infinity with speed $-\frac{1+a}{a}$, so $F(1/z) \sim z^{\frac{a+1}{a}}$ as $z\to 0$. Then $H(0)=0$, and since $F$ is conformal, $H$ is conformal in a small half-disc around 0 with radius $\epsilon>0$. Furthermore, $H$ maps the straight segment $[-\epsilon,\epsilon]$ into another straight segment. By the Schwarz reflection principle, $H$ can be extended analytically to the whole neighborhood of radius $\epsilon$ around $0$. This means that $$H(z) = z\psi_1(z),$$ where $\psi_1$ is a function that is holomorphic and does not vanish in the given neighborhood of $0$. We can then choose a branch of logarithm such that $$F\left(\frac1z\right) = z^{\frac{a+1}{a}}\psi_2(z),$$ with $\psi_2$ another regular function in the $\epsilon$ disc around 0. In particular, in a neighborhood of infinity, $F$ has a power series expansion of the form $$ z^{- \frac{a+1}{a}} \sum_{j=0}^\infty b_j z^{-j},$$ where $b_0\neq 0$. It follows that
    \begin{align*}
        \lim_{z\to\infty} \frac{F''(z)}{F'(z)} - \sum_{j=1}^3 \frac{\alpha_j-1}{z-z_j} &= \lim_{z\to\infty} \frac1z = 0.
    \end{align*}
    
    By Liouville's Theorem, $\frac{F''(z)}{F'(z)} - \sum_{j=1}^3 \frac{\alpha_j-1}{z-z_j}$ is constant and its value in the whole complex plane is $0$. We can write
    \begin{align*}
        (\log(F'(z)))' &= \frac{F''(z)}{F'(z)} = \sum_{j=1}^3 \frac{\alpha_j-1}{z-z_j},
        \intertext{which allows one to conclude}
        \log(F'(z)) &= \sum_{j=1}^3 \log(z-z_j)^{\alpha_j-1} + \tilde{C}.
        \intertext{Finally, integrating and replacing $z_1=-1, z_2=0, z_3=1$ and $\alpha_1 = 1/2, \alpha_2 = -1/a, \alpha_3 = 1/2$, leads to}
        F(z) = B + C\int_0^z\prod_{j=1}^3  (s-z_j)^{\alpha_j-1} \d s &= B + C\int_0^z s^{-\frac1a-1}(s-1)^{-\frac12}(s+1)^{-\frac12} \d s.
    \end{align*}
\end{proof}

\begin{proposition} \label{sc-mapping}
    For $a\in(-1,-1/2]$, the function $k(z) = -\frac1a\int_0^z \frac{1}{s^{1+\frac1a}\sqrt{1+s^2}}\d s$ is a conformal mapping between $\rightplane$ and $P$. 
\end{proposition}

\begin{proof}
For every $B,C\in \C$, \Cref{sc} guarantees that

\begin{equation*}
    \phi(z) = \phi(B,C;z) = B + C\int_0^z s^{-1-\frac1a}(s-1)^{-\frac12}(s+1)^{-\frac12}  \d s
\end{equation*}
is a conformal mapping from the upper half plane to $P$ up to a rotation, scaling and translation given by the constants $B,C$. The function $z \mapsto iz$ is conformal from $\rightplane$ to the upper half-plane $\C^+$, and the composition $\phi(iz)$ given by

\begin{equation*}
    \phi(iz) = B + C\int_0^{iz} s^{-1-\frac1a}(s-1)^{-\frac12}(s+1)^{-\frac12}  \d s = B + C\int_0^{z} i^{2+\frac1a}s^{-1-\frac1a}(s+i)^{-\frac12}(s-i)^{-\frac12}  \d s 
\end{equation*}

is conformal from $\rightplane$ to an affine transformation of $P$ for every $B,C$. In particular, $k$ is conformal between $\rightplane$ and an affine transformation $P'$ of $P$. The condition $k(0)=0$ already fixes the origin, so to determine uniquely $P'$ it is only necessary to compute $k(i)$ and $k(-i)$.

\begin{align*}
    k(-i) &= -\frac1a \int_0^{-i} s^{-1-\frac1a}(s^2 + 1)^{-\frac12} \d s = -\frac{\left(-1\right)^{-\frac1a}}a \int_0^{i}s^{-1-\frac1a}(s^2 + 1)^{-\frac12} \d s = e^{-i\frac\pi a}k(i),
\end{align*}
so that $k(-i)$ is a rotation in an angle of $-\frac{\pi}a$ of $k(i)$. Next,

\begin{align*}
    k(i) &= -\frac1a \int_0^i s^{-1-\frac1a}(s^2+1)^{-\frac12} \d s = -\frac{i^{-\frac1a}}a \int_0^1 s^{-1-\frac1a}(1-s^2)^{-\frac12} \d s\\
    &= -\frac{i^{-\frac1a}}a \int_0^1 s^{-\frac12-\frac1{2a}}(1-s)^{-\frac12} \left(\frac12s^{-\frac12}\right) \d s = -\frac{i^{-\frac1a}}{2a} \int_0^1 s^{-1-\frac1{2a}}(1-s)^{-\frac12} \d s\\
    &= -\frac{i^{-\frac1a}}{2a} B\left(-\frac{1}{2a}, \frac12\right),
\end{align*}
where $B(x,y)$ is the Beta function. Since $-1/(2a)$ and $1/2$ are both  positive, the Beta function is positive, and then $k(i)$ has argument $-\frac\pi{2a}$ and module \begin{equation}
    -\frac1{2a}B\left(-\frac{1}{2a}, \frac12\right). \label{eq:beta}
\end{equation}

This argument makes $P'$ symmetric with respect to the real axis.
\end{proof}

\begin{lemma} \label{lemma:extension}
    For every $a \in (-1,-1/2]$, the function $k^{-1}(z): P \to \rightplane$ admits an analytical continuation to the domain $D$ as defined in \Cref{analicity}.
\end{lemma}

\begin{proof}
    By \Cref{sc-mapping}, $k:\rightplane \to P$ is a conformal mapping. It is possible to distinguish two cases:

\begin{enumerate}[i.]
    \item \label{item:case-12} $a=-\frac12$. $k^{-1}$ is analytic in $\{ z : \Re(z) > -2 \}\setminus \R^-$, since $k(i)=k(-i)=-2$. The preimage of $\rightplane$ under $k^{-1}$ contains $\rightplane$ and all of the circular sectors around 0 with angle $\pi<\theta<2\pi$ and radius $0<r<2$. See \Cref{fig:P_a-12}.
    
    \item \label{item:case-1-12} $a\in(-1,-\frac12)$. The preimage of $\rightplane$ under $k^{-1}$ does not contain $\rightplane$, but it does contain all circular sectors around 0 with angle $\pi<\theta<-\frac\pi a$, and radius $\abs{k(i)}$. See \Cref{fig:P_apt-12}.
\end{enumerate}

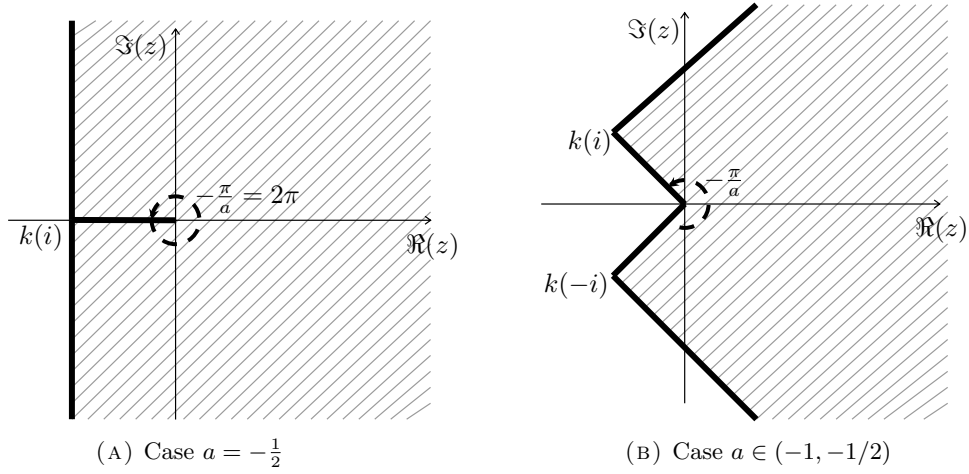
\begin{figure}
    \centering
    \begin{subfigure}{0.5\textwidth}
    \centering
        \tikzset{every picture/.style={line width=0.45pt}} %set default line width to 0.75pt        

\begin{tikzpicture}[scale=0.4,x=0.75pt,y=0.75pt,yscale=-1,xscale=1]
\path (0,443); %set diagram left start at 0, and has height of 543

%Straight Lines [id:da634688010381601] 
\draw [color={rgb, 255:red, 155; green, 155; blue, 155 }  ,draw opacity=1 ]   (200,110) -- (310,10) ;
%Straight Lines [id:da06192034808270941] 
\draw [color={rgb, 255:red, 155; green, 155; blue, 155 }  ,draw opacity=1 ]   (200,130) -- (330,10) ;
%Straight Lines [id:da7552804117094719] 
\draw [color={rgb, 255:red, 155; green, 155; blue, 155 }  ,draw opacity=1 ]   (200,150) -- (350,10) ;
%Straight Lines [id:da21542792770949615] 
\draw [color={rgb, 255:red, 155; green, 155; blue, 155 }  ,draw opacity=1 ]   (200,170) -- (370,10) ;
%Straight Lines [id:da5924957362728022] 
\draw [color={rgb, 255:red, 155; green, 155; blue, 155 }  ,draw opacity=1 ]   (200,190) -- (390,10) ;
%Straight Lines [id:da6836420534243748] 
\draw [color={rgb, 255:red, 155; green, 155; blue, 155 }  ,draw opacity=1 ]   (200,210) -- (410,10) ;
%Straight Lines [id:da25606869728310777] 
\draw [color={rgb, 255:red, 155; green, 155; blue, 155 }  ,draw opacity=1 ]   (200,230) -- (430,10) ;
%Straight Lines [id:da3851367840179094] 
\draw [color={rgb, 255:red, 155; green, 155; blue, 155 }  ,draw opacity=1 ]   (200,250) -- (450,10) ;
%Straight Lines [id:da4650281124250566] 
\draw [color={rgb, 255:red, 155; green, 155; blue, 155 }  ,draw opacity=1 ]   (200,270) -- (470,10) ;
%Straight Lines [id:da5380538510445385] 
\draw [color={rgb, 255:red, 155; green, 155; blue, 155 }  ,draw opacity=1 ]   (200,290) -- (490,10) ;
%Straight Lines [id:da15618811182767345] 
\draw [color={rgb, 255:red, 155; green, 155; blue, 155 }  ,draw opacity=1 ]   (200,310) -- (510,10) ;
%Straight Lines [id:da7135771585481248] 
\draw [color={rgb, 255:red, 155; green, 155; blue, 155 }  ,draw opacity=1 ]   (200,330) -- (530,10) ;
%Straight Lines [id:da6489885690804919] 
\draw [color={rgb, 255:red, 155; green, 155; blue, 155 }  ,draw opacity=1 ]   (200,350) -- (550,10) ;
%Straight Lines [id:da5054678031589592] 
\draw [color={rgb, 255:red, 155; green, 155; blue, 155 }  ,draw opacity=1 ]   (200,370) -- (570,10) ;
%Straight Lines [id:da14297654169793061] 
\draw [color={rgb, 255:red, 155; green, 155; blue, 155 }  ,draw opacity=1 ]   (200,390) -- (590,10) ;
%Straight Lines [id:da891424034966864] 
\draw [color={rgb, 255:red, 155; green, 155; blue, 155 }  ,draw opacity=1 ]   (200,410) -- (610,10) ;
%Straight Lines [id:da41728114114891846] 
\draw [color={rgb, 255:red, 155; green, 155; blue, 155 }  ,draw opacity=1 ]   (200,430) -- (630,10) ;
%Straight Lines [id:da7973690568593408] 
\draw [color={rgb, 255:red, 155; green, 155; blue, 155 }  ,draw opacity=1 ]   (200,450) -- (650,10) ;
%Straight Lines [id:da046961195777751485] 
\draw [color={rgb, 255:red, 155; green, 155; blue, 155 }  ,draw opacity=1 ]   (200,470) -- (650,30) ;
%Straight Lines [id:da8222615589595424] 
\draw [color={rgb, 255:red, 155; green, 155; blue, 155 }  ,draw opacity=1 ]   (200,490) -- (650,50) ;
%Straight Lines [id:da08114105086716827] 
\draw [color={rgb, 255:red, 155; green, 155; blue, 155 }  ,draw opacity=1 ]   (200,510) -- (650,70) ;
%Straight Lines [id:da14952644565382034] 
\draw [color={rgb, 255:red, 155; green, 155; blue, 155 }  ,draw opacity=1 ]   (220,510) -- (650,90) ;
%Straight Lines [id:da5473063802554726] 
\draw [color={rgb, 255:red, 155; green, 155; blue, 155 }  ,draw opacity=1 ]   (240,510) -- (650,110) ;
%Straight Lines [id:da1435248652407941] 
\draw [color={rgb, 255:red, 155; green, 155; blue, 155 }  ,draw opacity=1 ]   (260,510) -- (650,130) ;
%Straight Lines [id:da15366192578614524] 
\draw [color={rgb, 255:red, 155; green, 155; blue, 155 }  ,draw opacity=1 ]   (280,510) -- (650,150) ;
%Straight Lines [id:da871932171410194] 
\draw [color={rgb, 255:red, 155; green, 155; blue, 155 }  ,draw opacity=1 ]   (300,510) -- (650,170) ;
%Straight Lines [id:da06003778059634479] 
\draw [color={rgb, 255:red, 155; green, 155; blue, 155 }  ,draw opacity=1 ]   (320,510) -- (650,190) ;
%Straight Lines [id:da8710096802891837] 
\draw [color={rgb, 255:red, 155; green, 155; blue, 155 }  ,draw opacity=1 ]   (340,510) -- (650,210) ;
%Straight Lines [id:da14724815031650396] 
\draw [color={rgb, 255:red, 155; green, 155; blue, 155 }  ,draw opacity=1 ]   (360,510) -- (650,230) ;
%Straight Lines [id:da06218076169126263] 
\draw [color={rgb, 255:red, 155; green, 155; blue, 155 }  ,draw opacity=1 ]   (380,510) -- (650,250) ;
%Straight Lines [id:da8950752335967627] 
\draw [color={rgb, 255:red, 155; green, 155; blue, 155 }  ,draw opacity=1 ]   (400,510) -- (650,270) ;
%Straight Lines [id:da7964982972865481] 
\draw [color={rgb, 255:red, 155; green, 155; blue, 155 }  ,draw opacity=1 ]   (420,510) -- (650,290) ;
%Straight Lines [id:da16422098793191786] 
\draw [color={rgb, 255:red, 155; green, 155; blue, 155 }  ,draw opacity=1 ]   (440,510) -- (650,310) ;
%Straight Lines [id:da3529734969790512] 
\draw [color={rgb, 255:red, 155; green, 155; blue, 155 }  ,draw opacity=1 ]   (460,510) -- (650,330) ;
%Straight Lines [id:da6691425452510807] 
\draw [color={rgb, 255:red, 155; green, 155; blue, 155 }  ,draw opacity=1 ]   (480,510) -- (650,350) ;
%Straight Lines [id:da597710710701585] 
\draw [color={rgb, 255:red, 155; green, 155; blue, 155 }  ,draw opacity=1 ]   (500,510) -- (650,370) ;
%Straight Lines [id:da4346307684250741] 
\draw [color={rgb, 255:red, 155; green, 155; blue, 155 }  ,draw opacity=1 ]   (520,510) -- (650,390) ;
%Straight Lines [id:da6101365245815698] 
\draw [color={rgb, 255:red, 155; green, 155; blue, 155 }  ,draw opacity=1 ]   (540,510) -- (650,410) ;
%Straight Lines [id:da9660870987772634] 
\draw [color={rgb, 255:red, 155; green, 155; blue, 155 }  ,draw opacity=1 ]   (200,90) -- (290,10) ;
%Straight Lines [id:da9468291955855882] 
\draw [color={rgb, 255:red, 155; green, 155; blue, 155 }  ,draw opacity=1 ]   (200,70) -- (270,10) ;
%Straight Lines [id:da36964110731359956] 
\draw [color={rgb, 255:red, 155; green, 155; blue, 155 }  ,draw opacity=1 ]   (200,50) -- (250,10) ;
%Straight Lines [id:da5962151740417437] 
\draw [color={rgb, 255:red, 155; green, 155; blue, 155 }  ,draw opacity=1 ]   (200,30) -- (230,10) ;
%Straight Lines [id:da15856922223002656] 
\draw [color={rgb, 255:red, 155; green, 155; blue, 155 }  ,draw opacity=1 ]   (560,510) -- (650,430) ;
%Straight Lines [id:da3478350926924778] 
\draw [color={rgb, 255:red, 155; green, 155; blue, 155 }  ,draw opacity=1 ]   (580,510) -- (650,450) ;
%Straight Lines [id:da20111917549588632] 
\draw [color={rgb, 255:red, 155; green, 155; blue, 155 }  ,draw opacity=1 ]   (600,510) -- (650,470) ;
%Straight Lines [id:da119398756282875] 
\draw [color={rgb, 255:red, 155; green, 155; blue, 155 }  ,draw opacity=1 ]   (620,510) -- (650,490) ;
%Shape: Axis 2D [id:dp5066390813843493] 
\draw  (120,260.4) -- (650.07,260.4)(330.07,20.4) -- (330.07,510.4) (643.07,255.4) -- (650.07,260.4) -- (643.07,265.4) (325.07,27.4) -- (330.07,20.4) -- (335.07,27.4)  ;
%Straight Lines [id:da5371350273505343] 
\draw [line width=2.25]    (200,260) -- (200,10) ;
%Straight Lines [id:da6111319348969347] 
\draw [line width=2.25]    (200,260) -- (330.07,260.4) ;
%Straight Lines [id:da7559423641764833] 
\draw [line width=2.25]    (200,260) -- (200,510) ;
%Shape: Arc [id:dp38624583108693633] 
\draw  [draw opacity=0][dash pattern={on 5.63pt off 4.5pt}][line width=1.5]  (300.07,259.81) .. controls (300.39,243.51) and (313.7,230.4) .. (330.07,230.4) .. controls (346.64,230.4) and (360.07,243.83) .. (360.07,260.4) .. controls (360.07,276.97) and (346.64,290.4) .. (330.07,290.4) .. controls (314.32,290.4) and (301.41,278.27) .. (300.17,262.85) -- (330.07,260.4) -- cycle ; \draw [dash pattern={on 5.63pt off 4.5pt}][line width=1.5]  [dash pattern={on 5.63pt off 4.5pt}]  (300.43,255.73) .. controls (302.67,241.38) and (315.09,230.4) .. (330.07,230.4) .. controls (346.64,230.4) and (360.07,243.83) .. (360.07,260.4) .. controls (360.07,276.97) and (346.64,290.4) .. (330.07,290.4) .. controls (314.32,290.4) and (301.41,278.27) .. (300.17,262.85) ;  \draw [shift={(300.07,259.81)}, rotate = 282.98] [fill={rgb, 255:red, 0; green, 0; blue, 0 }  ][dash pattern={on 3.49pt off 4.5pt}][line width=0.08]  [draw opacity=0] (13.4,-6.43) -- (0,0) -- (13.4,6.44) -- (8.9,0) -- cycle    ;

% Text Node
\draw (351.07,212.8) node [anchor=north west][inner sep=0.75pt]    {$-\frac{\pi }{a} = 2\pi$};
% Text Node
\draw (130,262.4) node [anchor=north west][inner sep=0.75pt]    {$k( i)$};
% Text Node
\draw (616.07,272.8) node [anchor=north west][inner sep=0.75pt]    {$\Re ( z)$};
% Text Node
\draw (250,23.8) node [anchor=north west][inner sep=0.75pt]    {$\Im ( z)$};

\end{tikzpicture}
        \caption{Case $a = -\frac12$}
        \label{fig:P_a-12}
    \end{subfigure}
    \begin{subfigure}{0.45\textwidth}
    \centering
        \tikzset{every picture/.style={line width=0.45pt}} %set default line width to 0.75pt        

\begin{tikzpicture}[scale=0.4,x=0.75pt,y=0.75pt,yscale=-1,xscale=1]
%uncomment if require: \path (0,539); %set diagram left start at 0, and has height of 539

%Straight Lines [id:da6164183201408917] 
\draw [color={rgb, 255:red, 155; green, 155; blue, 155 }  ,draw opacity=1 ]   (250,170) -- (440,0) ;
%Straight Lines [id:da10513537560446617] 
\draw [color={rgb, 255:red, 155; green, 155; blue, 155 }  ,draw opacity=1 ]   (260,180) -- (460,0) ;
%Straight Lines [id:da5072668538716969] 
\draw [color={rgb, 255:red, 155; green, 155; blue, 155 }  ,draw opacity=1 ]   (270,190) -- (480,0) ;
%Straight Lines [id:da12526182632750404] 
\draw [color={rgb, 255:red, 155; green, 155; blue, 155 }  ,draw opacity=1 ]   (280,200) -- (500,0) ;
%Straight Lines [id:da24613815358878421] 
\draw [color={rgb, 255:red, 155; green, 155; blue, 155 }  ,draw opacity=1 ]   (290,210) -- (520,0) ;
%Straight Lines [id:da7281289018781257] 
\draw [color={rgb, 255:red, 155; green, 155; blue, 155 }  ,draw opacity=1 ]   (300,220) -- (540,0) ;
%Straight Lines [id:da9721215358303743] 
\draw [color={rgb, 255:red, 155; green, 155; blue, 155 }  ,draw opacity=1 ]   (310,230) -- (560,0) ;
%Straight Lines [id:da32014088496234305] 
\draw [color={rgb, 255:red, 155; green, 155; blue, 155 }  ,draw opacity=1 ]   (320,240) -- (580,0) ;
%Straight Lines [id:da8928901230635273] 
\draw [color={rgb, 255:red, 155; green, 155; blue, 155 }  ,draw opacity=1 ]   (330,250) -- (600,0) ;
%Straight Lines [id:da4547847799246547] 
\draw [color={rgb, 255:red, 155; green, 155; blue, 155 }  ,draw opacity=1 ]   (250,350) -- (620,0) ;
%Straight Lines [id:da21987732825304462] 
\draw [color={rgb, 255:red, 155; green, 155; blue, 155 }  ,draw opacity=1 ]   (260,360) -- (640,0) ;
%Straight Lines [id:da3537608093251172] 
\draw [color={rgb, 255:red, 155; green, 155; blue, 155 }  ,draw opacity=1 ]   (270,370) -- (660,0) ;
%Straight Lines [id:da6589709782181952] 
\draw [color={rgb, 255:red, 155; green, 155; blue, 155 }  ,draw opacity=1 ]   (280,380) -- (660,20) ;
%Straight Lines [id:da7447762041371727] 
\draw [color={rgb, 255:red, 155; green, 155; blue, 155 }  ,draw opacity=1 ]   (290,390) -- (660,40) ;
%Straight Lines [id:da8255898035304922] 
\draw [color={rgb, 255:red, 155; green, 155; blue, 155 }  ,draw opacity=1 ]   (300,400) -- (660,60) ;
%Straight Lines [id:da704507653649542] 
\draw [color={rgb, 255:red, 155; green, 155; blue, 155 }  ,draw opacity=1 ]   (310,410) -- (660,80) ;
%Straight Lines [id:da41177110825417507] 
\draw [color={rgb, 255:red, 155; green, 155; blue, 155 }  ,draw opacity=1 ]   (320,420) -- (660,100) ;
%Straight Lines [id:da8180458653129276] 
\draw [color={rgb, 255:red, 155; green, 155; blue, 155 }  ,draw opacity=1 ]   (330,430) -- (660,120) ;
%Straight Lines [id:da8622043876548691] 
\draw [color={rgb, 255:red, 155; green, 155; blue, 155 }  ,draw opacity=1 ]   (340,440) -- (660,140) ;
%Straight Lines [id:da8521963347622227] 
\draw [color={rgb, 255:red, 155; green, 155; blue, 155 }  ,draw opacity=1 ]   (350,450) -- (660,160) ;
%Straight Lines [id:da8781067580753334] 
\draw [color={rgb, 255:red, 155; green, 155; blue, 155 }  ,draw opacity=1 ]   (360,460) -- (660,180) ;
%Straight Lines [id:da3781601925385999] 
\draw [color={rgb, 255:red, 155; green, 155; blue, 155 }  ,draw opacity=1 ]   (370,470) -- (660,200) ;
%Straight Lines [id:da9770700452937056] 
\draw [color={rgb, 255:red, 155; green, 155; blue, 155 }  ,draw opacity=1 ]   (380,480) -- (660,220) ;
%Straight Lines [id:da5503348463115195] 
\draw [color={rgb, 255:red, 155; green, 155; blue, 155 }  ,draw opacity=1 ]   (390,490) -- (660,240) ;
%Straight Lines [id:da8657995733718264] 
\draw [color={rgb, 255:red, 155; green, 155; blue, 155 }  ,draw opacity=1 ]   (400,500) -- (660,260) ;
%Straight Lines [id:da6078910906401698] 
\draw [color={rgb, 255:red, 155; green, 155; blue, 155 }  ,draw opacity=1 ]   (410,510) -- (660,280) ;
%Straight Lines [id:da006498615893556603] 
\draw [color={rgb, 255:red, 155; green, 155; blue, 155 }  ,draw opacity=1 ]   (420,520) -- (660,300) ;
%Straight Lines [id:da7565857006801302] 
\draw [color={rgb, 255:red, 155; green, 155; blue, 155 }  ,draw opacity=1 ]   (440,520) -- (660,320) ;
%Straight Lines [id:da21903547303408621] 
\draw [color={rgb, 255:red, 155; green, 155; blue, 155 }  ,draw opacity=1 ]   (460,520) -- (660,340) ;
%Straight Lines [id:da12280220338390047] 
\draw [color={rgb, 255:red, 155; green, 155; blue, 155 }  ,draw opacity=1 ]   (480,520) -- (660,360) ;
%Straight Lines [id:da3292589837640829] 
\draw [color={rgb, 255:red, 155; green, 155; blue, 155 }  ,draw opacity=1 ]   (500,520) -- (660,380) ;
%Straight Lines [id:da9423873648370615] 
\draw [color={rgb, 255:red, 155; green, 155; blue, 155 }  ,draw opacity=1 ]   (520,520) -- (660,400) ;
%Straight Lines [id:da6633670094954679] 
\draw [color={rgb, 255:red, 155; green, 155; blue, 155 }  ,draw opacity=1 ]   (540,520) -- (660,420) ;
%Straight Lines [id:da8939464009932518] 
\draw [color={rgb, 255:red, 155; green, 155; blue, 155 }  ,draw opacity=1 ]   (560,520) -- (660,440) ;
%Straight Lines [id:da7635696067894008] 
\draw [color={rgb, 255:red, 155; green, 155; blue, 155 }  ,draw opacity=1 ]   (580,520) -- (660,460) ;
%Straight Lines [id:da8699303314892972] 
\draw [color={rgb, 255:red, 155; green, 155; blue, 155 }  ,draw opacity=1 ]   (605,520) -- (660,480) ;
%Straight Lines [id:da8144355989099584] % Extremo
\draw [color={rgb, 255:red, 155; green, 155; blue, 155 }  ,draw opacity=1 ]   (630,520) -- (660,500) ;
%Shape: Axis 2D [id:dp21064666550375022] 
\draw  (150,250) -- (650,250)(330,10) -- (330,500) (643,245) -- (650,250) -- (643,255) (325,17) -- (330,10) -- (335,17)  ;
%Straight Lines [id:da08285664384672031] 
\draw [line width=2.25]    (240,160) -- (330,250) ;
%Straight Lines [id:da02472931755585772] 
\draw [line width=2.25]    (240,160) -- (420,0) ;
%Straight Lines [id:da9371743469659536] 
\draw [line width=2.25]    (240,340) -- (330,250) ;
%Straight Lines [id:da3173904434500737] 
\draw [line width=2.25]    (240,340) -- (420,520) ;
%Shape: Arc [id:dp6003852621026782] 
\draw  [draw opacity=0][dash pattern={on 5.63pt off 4.5pt}][line width=1.5]  (308.97,228.61) .. controls (314.38,223.28) and (321.81,220) .. (330,220) .. controls (346.57,220) and (360,233.43) .. (360,250) .. controls (360,266.57) and (346.57,280) .. (330,280) .. controls (321.98,280) and (314.7,276.86) .. (309.32,271.73) -- (330,250) -- cycle ; \draw [dash pattern={on 5.63pt off 4.5pt}][line width=1.5]  [dash pattern={on 5.63pt off 4.5pt}]  (312.08,225.94) .. controls (317.08,222.21) and (323.28,220) .. (330,220) .. controls (346.57,220) and (360,233.43) .. (360,250) .. controls (360,266.57) and (346.57,280) .. (330,280) .. controls (321.98,280) and (314.7,276.86) .. (309.32,271.73) ;  \draw [shift={(308.97,228.61)}, rotate = 326.98] [fill={rgb, 255:red, 0; green, 0; blue, 0 }  ][dash pattern={on 3.49pt off 4.5pt}][line width=0.08]  [draw opacity=0] (13.4,-6.43) -- (0,0) -- (13.4,6.44) -- (8.9,0) -- cycle    ;

% Text Node
\draw (351,202.4) node [anchor=north west][inner sep=0.75pt]    {$-\frac{\pi }{a}$};
% Text Node
\draw (180,152.4) node [anchor=north west][inner sep=0.75pt]    {$k( i)$};
% Text Node
\draw (150,332.4) node [anchor=north west][inner sep=0.75pt]    {$k( -i)$};
% Text Node
\draw (616,262.4) node [anchor=north west][inner sep=0.75pt]    {$\Re ( z)$};
% Text Node
\draw (255,13.4) node [anchor=north west][inner sep=0.75pt]    {$\Im ( z)$};

\end{tikzpicture}
        \caption{Case $a \in (-1,-1/2)$}
        \label{fig:P_apt-12}
    \end{subfigure}
    \caption{Example of polygons $P$ for $a=-1/2$ and $a \in (-1,-1/2)$. When $a=-1/2$, $P$ covers the entire right half plane $\rightplane$, but this is not the case when $a<-1/2$.}
    \label{fig:polygons}
\end{figure}
    
    In case \ref{item:case-12}, for any $r$ in $(0,\abs{k(i)})$ and any $\theta$ in $(\pi,2\pi)$, \Cref{sc-mapping} shows that $k^{-1}$ is holomorphic in $\rightplane \cup \Theta_0(r,\theta)$, as defined in \Cref{analicity}.
    
    In case \ref{item:case-1-12}, if $r(a)\in (0,\abs{k(i)})$, and $\theta(a)\in (0,-\pi/a)$, \Cref{sc-mapping} does guarantee the analyticity in $\Theta_0(r(a), \theta(a))$, but not in $\rightplane$. To extend $k^{-1}$ analytically to the whole half-plane, take $k(z) = \phi(iz)$ with $\phi$ a conformal mapping between the upper half-plane and $P$ as in \Cref{sc}.  Since $\phi^{-1}$ is real valued and continuous at the border of $P$, except possibly at the vertices $0,k(i),k(-i)$, it is possible to apply the Schwarz reflection principle to extend $\phi$ analytically to the reflection of any subset of $P$ with respect to one of the sides (see for example \cite[page 298]{lang2013complex}).
    
    The whole right plane can be covered after a finite number of reflections of $P$ along its sides. Let $S_1, S_2, R_1, R_2$ be defined as in \Cref{sc}. A reflection of $P$ along $R_1$ results in $P_1 = \gamma(P)$ with $\gamma$ a rotation with respect to a fixed point $z_0$. Since the angle between $R_1$ and $R_2$ is $\left( -\frac1 a - 1  \right)\pi$, the angle between $\R$ and $R_2'$, the external side of the reflection, will be $3\left( -\frac1 a - 1  \right)\frac{\pi}2$. If $a \in [-3/4,-1/2)$, this angle will be greater or equal than $\frac\pi2$. In this case, a single reflection along $R_1$ allows one to extend analytically to the upper right quarter-plane. By symmetry, the same procedure can be done for the lower right quarter-plane. If $a \in (-1,-3/4)$, a single reflection of $P$ is not enough to cover the whole $\rightplane$ (see \Cref{fig:rotation}), but the procedure can be repeated along $R_1'$ and $R_2'$ the reflected sides. To cover the entire upper right quarter-plane $n = \lceil -\frac{1+2a}{2+2a} \rceil$ copies of $P$ are necessary. If they are named successively as $P_1, P_2, \cdots, P_n$, they correspond to consequent rotations:
    
    \begin{align*}
        P_j = \underbrace{\gamma(\gamma( \cdots \gamma(\gamma}_\text{$j$ times}(P) ))).
    \end{align*}
    To extend $\phi^{-1}$ analytically, define for $z \in P_j$
    \begin{align*}
        \phi^{-1}(z) = \left\{ \begin{array}{cc}
            \overline{\phi^{-1}(z)} & \text{if $j$ odd,} \\
            \phi^{-1}(z) & \text{if $j$ even.}
        \end{array} \right.
    \end{align*}
    
    \begin{figure}[!h]
        \centering
        % PARAMETERS
\newcommand{\axeslen}{5cm}
\newcommand{\xdir}{-0.1cm}
\newcommand{\ydir}{0.5cm}

% For a sector opening to the left, center it around 180°.
\begin{tikzpicture}[>=latex, thick, scale=0.5]%[scale=0.50]

% unified color fill
\definecolor{fillcolor}{RGB}{220,220,255}

% Polygon
\coordinate [] (O) at (0, 0);
\coordinate [label={ left:$\phi(1)$}] (ki) at (\xdir, \ydir);
\coordinate [label={ left:$\phi(-1)$}] (k_i) at (\xdir, -\ydir);
\coordinate [] (cor1) at (2*\axeslen, {2.52} );
\coordinate [] (cor2) at (2*\axeslen, {-2.52} );
\draw [very thick] (cor1) -- (ki) -- (O) -- (k_i) -- (cor2);
\fill[pattern={Lines[angle=45,distance={4pt},line width={0.01mm}]},pattern color=gray] (cor1) -- (ki) -- (O) -- (k_i) -- (cor2) -- (cor1);

% Axes only
\draw[-to] (-0.5*\axeslen,0) -- (2*\axeslen+0.3,0);
\draw[-to] (0,-0.45*\axeslen) -- (0,1.5*\axeslen+0.3);

% First rotation
\pgfexttransformmirror{\pgfpointxy{-0.1}{0.5}}{\pgfpointxy{10}{2.52}}
\draw[] (10, 2.52) -- (-0.1, 0.5) -- (0,0) -- (-0.1, -0.5) -- (10, -2.52);

% Second rotation
\pgfexttransformmirror{\pgfpointxy{-0.1}{-0.5}}{\pgfpointxy{10}{-2.52}}
\draw[] (6, 1.72) -- (-0.1, 0.5) -- (0,0) -- (-0.1, -0.5) -- (10, -2.52);

% Undo rotations 
\pgfexttransformmirror{\pgfpointxy{-0.1}{-0.5}}{\pgfpointxy{10}{-2.52}}
\pgfexttransformmirror{\pgfpointxy{-0.1}{0.5}}{\pgfpointxy{10}{2.52}}

% Annotations
\draw (-1, \axeslen) node [anchor=north west][inner sep=0.75pt] {$\Im(z)$};
\draw (2*\axeslen,-0.5) node [anchor=south west][inner sep=0.75pt] {$\Re(z)$};
\draw (4,2.5) node [anchor=south west][inner sep=0.75pt] {$P_1$};
\draw (0.35*\axeslen,4) node [anchor=south west][inner sep=0.75pt] {$P_2$};
\draw (5,0.5) node [][inner sep=0.75pt] {$P$};

\end{tikzpicture}
        \caption{Two successive reflections of $P$ along the upper side for $a\in (-1,-3/4)$. We define the value of $\phi^{-1}$ to be equal to $\overline{\phi^{-1}(\gamma^{-1}(z))}$ in $P_1$, and $\phi^{-1}(\gamma^{-1}(z))$ in $P_2$.}
        \label{fig:rotation}
    \end{figure}
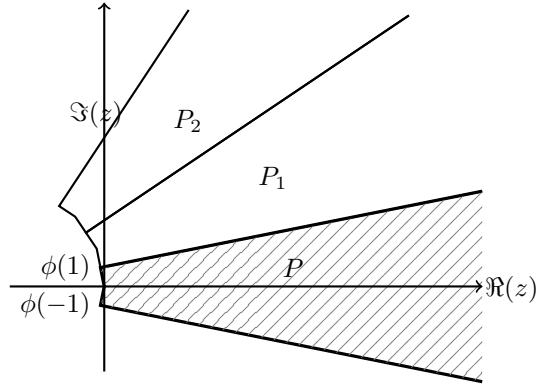
    
    Since $\phi^{-1}(z)$ is real valued at the border of $P$, its extension to each copy $P_j$ will be real valued at the border of $P_j$, and the Schwarz reflection principle assures that this extension is analytic in $\bigcup_j P_j$. Finally, the evaluation $k^{-1}(z) = -i\phi^{-1}(z)$ yields the analyticity of $k^{-1}$ in the extended domain.

\end{proof}

\begin{proof}[Proof of \Cref{result}]
The proof for the cases $a\in \{1,0,-1\}$ is immediate from \cref{eq:zero_and_one,eq:minus_one}.  

By \Cref{analicity}, for every $(a,x) \in \left((-1,1)\setminus \{0\}\right) \times \R_+$, the  function $A(k^{-1}(z))$ is holomorphic in the domain  $\Delta(\epsilon(a,x),\theta(a,x)) \eqqcolon \Delta_a$  and it is possible to apply \Cref{flajolet} to estimate its coefficients with its behaviour around the singularity at $0$. 

Let $a\in(-1,0)$. Taking any path going from 0 to $z$ and staying into $\Delta_a$, as $z\to0$:

\begin{align*}
    k(z) &= -\frac1a \int_0^z (s^2+1)^{-\frac12}s^{-1-\frac1a}\mathrm d s \sim -\frac1a\int_0^z s^{-1-\frac1a} \d s = z^{-\frac1a}.
\end{align*}
This approximation yields the estimation $k^{-1}(z) \sim z^{-a}$, as $z\to0$.

In the function $A(t)=\left( (1-2q)t + \sqrt{1+t^2} \right)t^\frac1a$, the term $(1-2q)t + \sqrt{1+t^2}$ is continuous at $t=0$, so as $t\to 0$, $(2q+1)t + \sqrt{1+t^2}\sim1$, while $t^{\frac1a} \to \infty$ as $t\to0$. By definition of $k$, $k^{-1}(0) = k^{-1}(k(0)) = 0$, so 0 is a singular point of $A(k^{-1}(z))$ and around this point

\begin{equation*}
    A(k^{-1}(z)) \sim \left( k^{-1}(z) \right)^\frac1a.
\end{equation*}
By the previous approximations, it follows that near $0$,
\begin{equation*}
     \left(k^{-1}(z)\right)^\frac1a \sim z^{-1}.
\end{equation*}
Since $L_a(x,z) = \sinh^{-\frac1a}(x) A(k^{-1}(k(\sinh(x)) + z\sinh^{-\frac1a}(x) ))$, then as $z \to k(\sinh(x)) \sinh^{\frac1a}(x)$
\begin{equation*}
    L_a(x,z) \sim \frac{\sinh^{-\frac1a}(x)}{k(\sinh(x))-z\sinh^{-\frac1a}(x)} =  \frac{\sinh^{-\frac1a}(x)}{k(\sinh(x))}\left( 1 - \frac{z\sinh^{-\frac1a}(x)}{k(\sinh(x))} \right)^{-1}.
\end{equation*}
Due to this behavior near the singularity, it is possible to use \Cref{flajolet} to conclude that 
\begin{equation*}
    [z^n]L_a(x,z) \sim \frac{\sinh^{\frac1a}(x)}{k(\sinh(x))}[z^n]\left( 1 - \frac{z\sinh^{-\frac1a}(x)}{k(\sinh(x))} \right)^{-1}, \qquad n\to\infty
\end{equation*}
and the right-hand satisfies
\begin{align*}
    [z^n]\left( 1 - \frac{z \sinh^{-\frac1a}(x)}{k(\sinh(x))} \right)^{-1} &= [z^n]\sum_{j=0}^{\infty} \left(\frac{z\sinh^{-\frac1a}}{k(\sinh(x))}\right)^j = \left(\frac{\sinh^{-\frac1a}(x)}{k(\sinh(x))}\right)^n.
\end{align*}
Which implies that, as $n \to \infty$
\begin{equation*}
    \mathcal L_{n+1}(a;x) = [z^n]L(x,z) \sim \frac{\sinh^{-\frac1a}(x)}{k(\sinh x)} \left( \frac{\sinh^{-\frac1a}(x)}{k(\sinh x)} \right)^n = \left( \frac{\sinh^{-\frac1a}(x)}{k(\sinh x)} \right)^{n+1}.
\end{equation*}

When $a>0$, $k$ is strictly decreasing in $(0,\infty)$. As $z$ goes to $\infty$,

\begin{align*}
     k(z) &= \frac1a \int_z^\infty (s^2+1)^{-\frac12}s^{-1-\frac1a}\mathrm d s = \frac1a\int_z^\infty (s^{-2}+1)^{-\frac12}s^{-2-\frac1a}\d s  \sim \frac{z^{-1-\frac1a}}{a+1},
\end{align*}
so the behavior of $k^{-1}$ near the origin is
\begin{align*}
     k^{-1}(z) &\sim (a+1)^{- \frac{a}{a+1}}z^{-\frac{a}{a+1}}, \qquad z \to 0.
\end{align*}
Since $-a/(a+1) <0$, in this case $k^{-1}(z) \to \infty$ as $z\to 0$, so
\begin{align*}
     A(k^{-1}(z)) &= \left( (1-2q) k^{-1}(z) + k^{-1}(z)\sqrt{1+(k^{-1}(z))^{-2}} \right)(k^{-1}(z))^\frac{1}{a}\\
     & \sim \left( (2-2q) k^{-1}(z) \right)(k^{-1}(z))^\frac{1}{a} \sim  2(1-q)\left( [(a+1)z]^{-\frac{a}{a+1}} \right)^{1+\frac1a} = \frac{2(1-q)}{a+1}z^{-1}, \quad z \to 0.
\end{align*}

Finally, by \Cref{flajolet}, when $z \to k(\sinh(x))\sinh^\frac1a(x)$,
\begin{align*}
    L_a(x,z) = \frac{A(k^{-1}( k(\sinh(x)) - z\sinh^{-\frac1a}(x) ))}{\sinh^{\frac1a}(x) } &\sim \frac{2(1-q)}{(a+1)\sinh^{\frac1a}(x)}\left( k(\sinh(x)) - z\sinh^{-\frac1a}(x) \right)^{-1}\\
    &=  \frac{2(1-q)}{a+1}\frac{\sinh^{-\frac1a(x)}}{k(\sinh(z))} \left( 1 - \frac{z\sinh^{-\frac1a}(x)}{k(\sinh(x))} \right)^{-1}.
\end{align*}

    And the estimate for the coefficients yields
\begin{align*}
    [z^n] L_a(x,z) &\sim  \frac{2(1-q)\sinh^{-\frac1a}(x)}{(a+1)k(\sinh(z))} \left( \frac{\sinh^{-\frac1a}(x)}{k(\sinh(x))} \right)^n \\
    &= \frac{2(1-q)}{a+1}\left( \frac{\sinh^{-\frac1a}(x)}{k(\sinh(x))} \right)^{n+1}.
\end{align*}

Finally, by \cref{eq:even} it is possible to replace $q$ by $1-q$ to find the estimates for $x<0$.
\end{proof}

\section{Further properties of the exponential growth and large deviations}\label{last_subsection}

This final subsection studies properties of the function $\rayon$ and derives the \gls{ldp} from the asymptotic behavior of the Laplace transform. The resulting \gls{ldp} coincides with that obtained in \cite{franchini2023large}.

 Remark that it is possible to find explicit expressions of $\rayon$ for particular values of $a$. Integration by parts yields a recursive formula for $\int \sinh^{-1-\frac1a}(s)\d s$:
 \begin{align*}
     \int \sinh^{-1-\frac1a}(s)\d s &= \frac{-a\sinh^{-2-\frac1a}(s)\cosh s}{a+1} - \frac{2a+1}{a+1}\int \sinh^{-\frac1a-3}(s) \d s.
 \end{align*}
 Adding the integration limits in \cref{eq:definition_fgd} the explicit forms of $\rayon = \frac1{\phi(a;x)}$ when $-1-\frac1a\in \mathbb Z$ follow from \cite[Section 2.4]{gradshteyn2014table}:
\begin{align*}
    \rho_{-1/(2m+1)}(x) &= \frac{2m+1}{\sinh^{2m+1}(\abs{x})} \left[ (-1)^m\binom{2m}{m}\frac{\abs{x}}{2^{2m}} + \frac{1}{2^{2m-1}}\sum_{k=0}^{m-1} (-1)^k \binom{2m}{k}\frac{\sinh\left( 2\abs{x}(m-k) \right)}{2(m-k)} \right],\\
    \rho_{-1/(2m+2)}(x) &= \frac{2m+2}{\sinh^{2m+2}(\abs{x})}\left[ (-1)^m \sum_{k=0}^m (-1)^k \binom{m}{k} \frac{\cosh^{2k+1}(x)}{2k+1} - (-1)^m \sum_{k=0}^m  \binom{m}{k} \frac{(-1)^k}{2k+1} \right],\\
    \rho_{1/(2m-1)}(x)  &= \cosh (\abs{x})\left[ 1 + \sum_{k=1}^{m-1} (-1)^{k} \frac{2^k (m-1)(m-2) \cdots (m-k)}{(2m-3)(2m-5)\cdots (2m-2k-1)} \sinh^{2k-1}(\abs{x}) + C(m) \right], \\
    \rho_{1/(2m)}(x)    &= \cosh (\abs{x})\left[ 1 + \sum_{k=1}^{m-1} (-1)^{k} \frac{(2m-1)(2m-3)\cdots (2m-2k+1)}{2^k (m-1)(m-2) \cdots (m-k)}\sinh^{2k}(x) \right]\\
    &\phantom{ooooo} + 2m\sinh^{2m}(x)\left[(-1)^{m-1}\frac{(2m-1)!!}{(2m)!!} \ln \tanh \left( \frac {\abs{x}} 2 \right) \right].
\end{align*}
The constant $C(m)$ is equal to $1$ if $m=1$ and 0 in any other case. In particular, explicit forms of $\rayon$ for some rational numbers are shown in \cref{tab:explicit_values}. The plots of these functions are also shown in \Cref{plots}. 

\begin{table}[!h] \caption{Explicit form of $\rayon$ for particular values of $a$.\label{tab:explicit_values}} 
\begin{tabular}{@{}ll@{}}
\toprule
Value of $a$ & $\rayon$ \\ \midrule
  -1 &  $x/\sinh(x)$ \\
 -1/2 & $2\left( \cosh (x) -1 \right)/\sinh^2(x)$ \\
  -1/3 & $3\sinh^{-3}(x)\left( - x/2 + \frac14\sinh(2x) \right)$ \\
 1/3 & $2 \sinh^3(\abs{x})+\cosh(x)-2 \sinh^{2}(x) \cosh(x)$ \\
 1/2 & $\cosh (x) + \sinh^2(x)\ln \tanh \left(  \left\lvert  x/2 \right\rvert \right)$ \\
 1 & $e^{-\lvert x \rvert}$ \\
  \bottomrule
\end{tabular}
\end{table}

\begin{figure}[!ht]
    \def\mathdefault#1{#1}\everymath=\expandafter{\the\everymath\displaystyle}
    \input{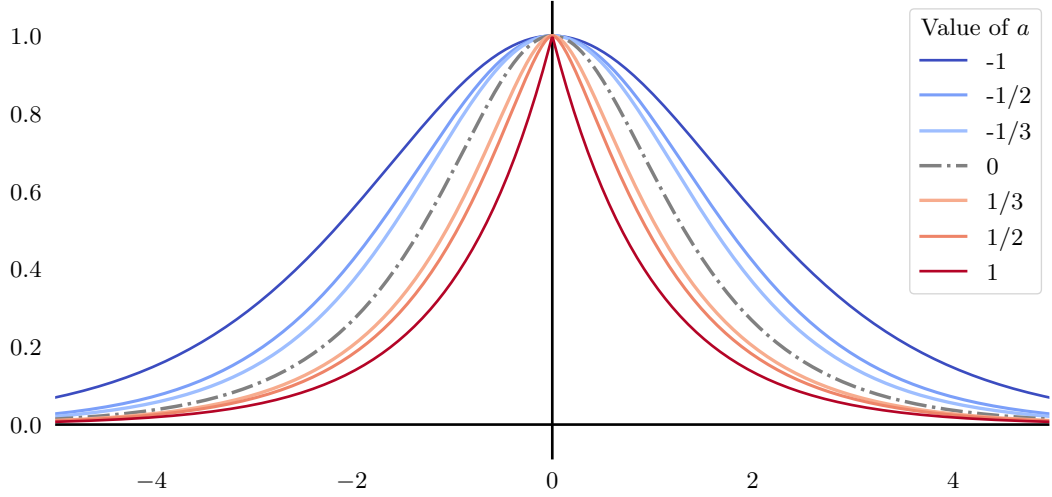}
    \caption{Plots of $\rayon$ for different values of $a$. For positive values of $a$, $\rayon <1/\cosh(x)$, while  ) for different values offor negative values $\rayon > 1/\cosh(x)$.}
    \label{plots}
\end{figure}

From both \Cref{tab:explicit_values} and \Cref{plots} it is possible to see that in the degenerate case $a=1$, $\rayon$ is not derivable in zero. For every other $a \in [-1,1)\setminus \{0\}$, although it is at least derivable at $0$, it is not always analytic.

\begin{proposition} \label{analyticity}
The function $x \mapsto \rayon$ is analytic at $0$ if $a \in [-1,0]$ and non-analytic if $a \in (0,1]$. 
\end{proposition}~Since $\rayon = 1/\phi(a;x)$, and none of the functions vanish for finite $x$, this result implies in particular \Cref{non_analytic}.

\begin{proof}
    Analyticity of $\rayon$ at $x=0$ for $a\in [-1,0)$ is immediate from \cref{eq:hypergeometric} and the fact that $\rayon = \sinh^{\frac1a}(\abs{x})k(\sinh( \abs{x}))$: 
\begin{align*}
     \rayon &= \frac{\sinh^{\frac1a}(\abs{x}) }{a} \sinh^{-\frac1a}(\abs{x}) {}_2 F_1\left[ 2,-\frac1{2a} ; 1 - \frac1{2a}; -\sinh^2(\abs{x}) \right]\\ &= \frac{ {}_2 F_1\left[ 2,-\frac1{2a} ; 1 - \frac1{2a}; -\sinh^2(x) \right] }a.
\end{align*}
To see the non analytic behavior at 0 in the case $a\in (0,1)$, let
\begin{align*}
    k(x) = \frac1a\int_x^\infty \frac{s^{-1-\frac1a}}{\sqrt{ 1 + s^2 }} \d s &= \frac1a\int_x^\infty s^{-1-\frac1a}\left( \sum_{n=0}^\infty \binom{-\frac12}{n} s^{2n} \right) \d s
\end{align*}
Take any $N\in \N$, as $x$ decreases to 0, $k$ can be approximated with the truncation of its Taylor polynomial up to an order $N$, i.e. $k(x) \sim \frac1a\int_x^\infty s^{-1-\frac1a}\sum_{n=0}^N \binom{-\frac12}{n} s^{2n}\d s$. This, in particular implies for $N = \left\lfloor \frac1{2a}\right\rfloor$
\begin{align*}
    \lim_{x\to 0^+} \left\lvert k(x) - \frac1a\int_x^\infty s^{-1-\frac1a}\sum_{n=0}^{\left\lfloor \frac1{2a}\right\rfloor} \binom{-\frac12}{n} s^{2n} \d s \right\rvert < \infty.
\end{align*}
Define
\begin{align*}
    C(a) \coloneqq \lim_{x\to 0^+} \left( k(x) - \frac1a\int_x^\infty s^{-1-\frac1a}\sum_{n=0}^{\left\lfloor \frac1{2a}\right\rfloor} \binom{-\frac12}{n} s^{2n} \d s \right) = \lim_{x\to0^+}\frac1a\int_x^\infty s^{-1-\frac1a}\sum_{ n = \left\lfloor \frac1{2a}\right\rfloor }^\infty \binom{-\frac12}{n} s^{2n} \d s.
\end{align*}
The coefficients in the series can be written as
\begin{align*}
    \binom{-\frac12}{n} = \frac{(-1/2)(-3/2)\cdots(-(2n-1)/2)}{n!} = (-1)^n \frac{(2n)!}{2^{2n} (n!)^2}.
\end{align*}
So they are alternating and monotonically decreasing, since
\begin{align*}
    \frac{\abs{ \binom{-1/2}{n} } }{ \abs{ \binom{-1/2}{n+1} } } = \frac{(2n)! 2^{2n + 2} (n+1)!^2 }{ 2^{2n} (n!)^2(2n+2)! } = \frac{4(n+1)^2}{(2n+1)(2n+2)} = \frac{2n+2}{2n+1}>1.
\end{align*}
By the alternating series estimation test, this means the Taylor remainder has a constant sign equal to the sign of the first neglected term. Write $g(s) = s^{-1-\frac1a}(1+s^2)^{-\frac12} - s^{-1-\frac1a}\sum_{n=0}^{ \left\lfloor 1/2a\right\rfloor } \binom{-1/2}{n} s^{2n}$. By the previous discussion, $g$ has a constant sign in $(0,\infty)$, so
\begin{align}
    C(a) = \frac1a\int_0^\infty g(s) \d s = \frac{\mathrm{sign}(g)}a\int_0^\infty s^{-1-\frac1a}\abs{g(s)} \d s. \label{definition_of_constant}
\end{align}
Thus, $C(a)\neq0$ since $s^{-1-\frac1a}\abs{g(s)} $ is non-negative and non-constant. Now let, for $0 < x < 1$,
\begin{align*}
    \int_{x}^\infty \frac{s^{-1-\frac1a}}{\sqrt{s^2+1}} \d s &= \int_x^1 \frac{s^{-1-\frac1a}}{\sqrt{s^2+1}} \d s + C_1(a).
\end{align*}
Suppose first that $1/(2a) \notin \mathbb Z$, then
\begin{align*}
    \int_x^1 \frac{s^{-1-\frac1a}}{\sqrt{s^2+1}} \d s &= \int_x^1 s^{-1-\frac1a}\left(\frac{1}{\sqrt{s^2+1}} - \sum_{n=0}^{ \left\lfloor \frac1{2a}\right\rfloor} \binom{-\frac12}{n} s^{2n}\right) \d s + C_2(a) + \sum_{n=0}^{\left\lfloor \frac1{2a}\right\rfloor} \binom{-\frac12}{n} \frac{ax^{2n-\frac1a}}{1-2na}.
\end{align*}
Take $C(a)$ as defined in \cref{definition_of_constant}. As $x$ tends to 0:
\begin{align*}
    k(x) \sim C(a) + \sum_{n=0}^{\left\lfloor \frac1{2a}\right\rfloor} \binom{-\frac12}{n} \frac{x^{2n-\frac1a}}{1-2na}.
\end{align*}

Now let
\begin{align*}
    \sinh^{\frac1a}(\abs{x}) &= \left( \sum_{n=0}^\infty \frac{ \abs{x}^{2n+1}}{(2n+1)!} \right)^\frac1a = \abs{x}^\frac1a \left( 1 + \sum_{n=1}^\infty \frac{\abs{x}^{2n}}{(2n+1)!} \right)^\frac1a = \abs{x}^{\frac1a}\left( 1 + c_2 x^2 + c_4 x^4 + \cdots \right),
\end{align*}
where $c_2 = \frac1{6a}$, and all of the $c_{2n}$ can be computed explicitly. As $x \to 0$
\begin{align*}
    \rayon &= \sinh^{\frac1a}(\abs{x})\left( C(a) + \sum_{n=0}^{\left\lfloor \frac1{2a}\right\rfloor} \binom{-\frac12}{n} \frac{\sinh^{2n-\frac1a}(\abs{x})}{1-2na} \right) \\
    &= C(a)\abs{x}^{\frac1a}\left( 1 + c_2 x^2 + c_4 x^4 + \cdots \right) + \sum_{n=0}^{ \left\lfloor \frac1{2a}\right\rfloor } \binom{-\frac12}{n}\frac{\sinh^{2n}(x)}{1-2na}.
\end{align*}
The expansion of $\rayon$ around zero has even terms of all orders less than $1/a$ and then a term $C(a)\abs{x}^{\frac1a}$. In the case that $1/(2a) \in \Z$:
\begin{align*}
    \int_x^1 \frac{s^{-1-\frac1a}}{\sqrt{s^2+1}} \d s &= C_3 + C_2(a) + \sum_{n=0}^{\left\lfloor \frac1{2a}\right\rfloor -1} \binom{-\frac12}{n} \frac{ax^{2n-\frac1a}}{1-2na} - \frac1a\binom{-\frac12}{\frac1{2a}}\ln(x).
\end{align*}
Then, near $0$
\begin{align*}
    \rayon &\sim \sinh^{\frac1a}(\abs{x})\left( C(a) + \sum_{n=0}^{\left\lfloor \frac1{2a}\right\rfloor -1} \binom{-\frac12}{n} \frac{\sinh(\abs{x})^{2n-\frac1a}}{1-2na} - \frac1a\binom{-\frac12}{\frac1{2a}}\ln(\sinh(\abs{x})) \right)\\
    &= C(a)x^{\frac1a}\left( 1 + c_2 x^2 + c_4 x^4 + \cdots \right) \\
    &\phantom{000000000}+ \sum_{n=0}^{ \left\lfloor \frac1{2a}\right\rfloor-1} \binom{-\frac12}{n}\frac{\sinh^{2n}(x)}{1-2na} -\frac1a\binom{-\frac12}{\frac1{2a}}\left( \ln \abs{x} + d_2x^2 + d_4x^4 + \cdots + \right)\sinh^{\frac1a}(x),
\end{align*}

for some $d_1, d_2, \dots$ In this case, the terms $\abs{x}^{\frac1a}$ are regular since $1/a$ is pair, but terms of the form $x^{2\ell+\frac1a}\ln \abs{x}$ appear in the expansion for every $\ell$ natural.
\end{proof}

\Cref{plots} suggests that the function $a \mapsto \rayon$ is monotone and that, for $a$ close to $0$, it is approximately $1/\cosh(x)$. This is indeed the case, and defining $\rho_0(x) \coloneqq 1/\cosh(x)$ ensures that $\rayon$ extends continuously to $a=0$.

\begin{proposition}
The function $a \mapsto \rayon$ is decreasing and it converges to $1/\cosh x$ when $a$ goes to 0.
\end{proposition}

\begin{proof}
Let first $a<0$. By definition of $\rayon$:
\begin{align}
    \rayon = - \frac{\sinh^{\frac1a}(\abs{x})}{a}\int_0^{\abs{x}} &\sinh^{-1-\frac1a}(s) \d s =  -\frac1{a\sinh \abs{x}}\int_0^{\abs{x}} \left( \frac{\sinh (\abs{x})}{\sinh (s)} \right)^{1+\frac1a} \d s. \notag
    \intertext{Take $u = \log \left( \frac{\sinh (\abs{x})}{ \sinh s} \right)$, so that $\frac{\d u}{\d s} = -e^u\sqrt{\sinh^2(x)e^{-2u} + 1}/\sinh(\abs{x})$, and }
    \rayon &= -\frac{1}{a}\int_0^\infty \frac{e^{\frac u a} \d u}{\sqrt{ e^{-2u}\sinh^2(x) + 1 } }. \notag
    \intertext{Integration by parts in the last expression yields}
    \rayon&= \frac1{\cosh (x)} + \sinh^2(x)\int_0^\infty \frac{e^{u\left( \frac1a+1 \right)} \d u}{( e^{2u} + \sinh^2 (x) )^{\frac32}}. \label{eq:continuity_a0}
\end{align}

The last integral is decreasing in $a$ for $a\in (-1,0)$, and for these values, the integrand is dominated by the integrable function $1/( e^{2u} + \sinh^2 x )^{3/2}$. The Dominated Convergence Theorem yields the convergence as $a$ increases to $0$. The proof for $a>0$ is analogous.
\end{proof}

The final part of this section presents the proof of \Cref{large_deviations} and its equivalence to the result in \cite{franchini2015large,franchini2023large}.

\begin{proof}[Proof of \Cref{large_deviations}]
Suppose $a\in [-1,0)$. By \Cref{result}, when $n$ goes to $+\infty$
\begin{align*}
    \L{n} &= C(a,q) \phi(a;x)^n + o( \phi(a;x)^n ),
\intertext{where $C(a,q) = (1-q)/(a+1)$ if $a>0, x>0$, $C(a,q) = q/(a+1)$ if $a>0, x<0$ and $C(a,q)\equiv 1$ if $a<0$. Then}
   \frac{\L{n}}{C(a,q)\phi(a;x)^n} &= 1 + o(1), \qquad n \to \infty.
   \intertext{Taking logarithm and dividing by $n$:}
   \frac1n\log\L{n} - \frac1n\log C(a,q) - \log \phi(a;x) &= \frac1n\log\left( 1 + o(1) \right) = o(n^{-1}).
\end{align*}
Which implies $$\lim_{n\to \infty} \frac1n\log\L{n} = \log \phi(a;x) \eqqcolon \Lambda(a;x).$$ This function is finite for all $x \in \R$, by \cref{eq:continuity_a0}. Deriving $\log \phi(a;x)$ and using \cref{eq:continuity_a0} shows that $\lim_{x\to \infty}\abs{ \Lambda'(a;x) } = \infty$, so the Gärtner-Ellis Theorem (see, e.g. \cite[page 44]{DemboZeitouni}), asserts that the \gls{erw} satisfies a \gls{ldp} with good rate function $$\Lambda^*(a;x) = \sup_{t \ge 0} \left\{ xt - \Lambda(a;t) \right\}.$$
\end{proof}

Remarkably, when $\fgd$ is the Laplace transform of some probability measures, the \gls{ldp} satisfied by the \gls{erw} is exactly the same as the one satisfied by an i.i.d. sample of random variables with Laplace transform $\fgd$. This is the case for some particular values of $a$. When $a=-1$, $\fgd=\sinh(x)/x$ is the Laplace transform of a uniform random variable in $[-1,1]$, and $\Lambda^*(a;x)$ can be written in terms of the inverse Langevin function $\mathrm{L}(x)$, \[\Lambda^*(-1;x) = x \mathrm{L}^{-1}(x) - \log \left( \frac{ \sinh\left( \mathrm{L}^{-1}(x)\right) }{\mathrm{L}^{-1}(x)} \right),\]
while for $a=-\frac12$, $\fgd$ is the Laplace transform of the mean of two i.i.d. symmetric Rademacher random variables, i.e. it is the Laplace transfomr of the measure $\frac12\delta_{-1} + \frac14\delta_0 + \frac12\delta_1$. The rate function in this case is \[ \Lambda^*(-1/2;x) = (1+x)\log(1+x) + (1-x)\log(1-x).\]

As mentioned in the introduction, a \gls{ldp} for an equivalent model has already been found in Franchini \cite{franchini2015large} and Franchini and Balzan \cite{franchini2023large}. To see the equivalence of this expression to the previously known, recall that the DEK model with sample size 1 and trust parameter $p$ as defined in \cite{franchini2023large} is equal to $(S_n +n)/2n$, where $S$ is an \gls{erw}.

By the definition of $\psi(x)$ in \cite{franchini2023large}:

\begin{align*}
    \psi(x) \coloneqq \lim_{n\to\infty} \frac1n \log \sum_{j=0}^n e^{x j} \mathbb P\left( \frac12 \left( \frac{S_n + n}{n} \right) = j/n \right) 
    &=  \lim_{n\to\infty} \frac1n \log e^{\frac{nx}{2}}\sum_{j=0}^n e^{\frac x2 (2j -n)} \mathbb P\left( S_n = 2j-n \right)\\
    &= \lim_{n\to\infty}\frac1n\log e^{\frac{nx}{2}} \mathbb E\left[ e^{\frac x2 S_n} \right]
    = \frac x2 + \Lambda(-x/2).
    \intertext{Thus, by symmetry of $\Lambda$, for $x >0$,}
    e^{-\frac x2} \rho_a\left( \frac x2 \right) =\exp\left(-\Lambda\left(\frac x2\right)- \frac x2\right) &= \exp\left(-\psi(x)\right). 
\end{align*}
This can also be verified by a direct computation. Take $a,x > 0$.
\begin{align*}
    e^{-\frac x2} \rho_a\left( \frac x2 \right) &= \frac{e^{-\frac x2} \sinh^{\frac1a}(x/2)}{a} \int_{\frac x2}^\infty \sinh^{-1-\frac1a}(t) \d t = \frac{e^{-\frac x 2} ( e^{\frac x2} - e^{-\frac x2} )^\frac1a}{a}\int_x^\infty (e^{\frac t2} - e^{-\frac t2})^{-1-\frac1a} \d t\\
    &=\frac{e^{\frac {1-a}{2a} x}\left(1- e^{-x}\right)^\frac1a}{a}\int_x^\infty e^{-t - \frac{1-a}{2a} t}(1-e^{-t})^{-1-\frac1a}\d s.
\intertext{Take the change of variable $t = 1-e^{-s}$,}
     e^{-\frac x2} \rho_a\left( \frac x2 \right) &=     \frac{e^{\frac{1-a}{2a} x}(1-e^{-x})^{\frac1a}}{a}\int_{1-e^{-x}}^1 (1-t)^{\frac{1-a}{2a} }t^{-\frac1a-1} \d t
\end{align*}
Using integration by parts in the last integral leads to
\begin{align*}
    \int_{1-e^{-x}}^1 (1-t)^{\frac{1-a}{2a} }t^{-\frac1a-1} \d t &= ae^{-x\frac{1-a}{2a}}(1-e^{-x})^{-\frac1a} - \frac{1-a}{2a} \int_{1-e^{-x}}^1 t^{-\frac1a}(1-t)^{\frac {1-a}{2a} -1} \d t.
    \intertext{Substituting back into $e^{-\frac x2} \rho_a\left( \frac x2 \right)$}
    e^{-\frac x2} \rho_a\left( \frac x2 \right) &= 1 - \frac{1-a}{2a} e^{\frac{1-a}{2a} x}(1-e^{-x})^{\frac1a}\int_{1-e^{-x}}^1 (1-t)^{\frac{1-a}{2a} }t^{-\frac1a} \d t = \exp(-\psi(x)),
    \intertext{which is exactly $\psi(x)$ for $a,x>0$ as in eq. (93) in \cite{franchini2023large}.}
\end{align*}

Since precise asymptotics of the Laplace transform yield the \gls{ldp}, the present approach naturally suggests extensions leading to sharper large deviation asymptotics. These questions will be addressed in future work.

\section*{Acknowledgements} The author gratefully  acknowledges his supervisors, Kilian Raschel and Pierre Tarrago, for the stimulating discussions, the time they devoted to reading the drafts, and their invaluable comments. The inspiration of this work stemmed from the valuable insights of Bernard Bercu, Michel Bonnefont, Hélène Guérin, Lucile Laulin, Kilian Raschel, and Adrien Richou in their discussions around the Elephant Random Walk and and precise asymptotics results for similar models. This project has received funding from the European Union’s Horizon Europe research and innovation programme under the Marie Skłodowska-Curie grant agreement n° 101126554.

\begin{center}
    \includegraphics[width=0.15\textwidth]{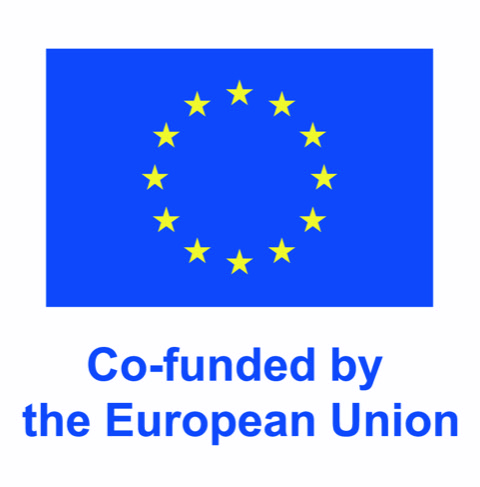}
\end{center}

\section*{Disclaimer}

Co-funded by the European Union. Views and opinions expressed are however those of the author only and do not necessarily reflect those of the European Union. Neither the European Union nor the granting authority can be held responsible for them.

\section{Appendix: Webpage with interactive implementation of the polygon} \label{interactive_polygon}

An interactive implementation of the polygons in \Cref{subsection_a-1-12} can be found in this link \url{https://soy-esaul.github.io/interactive_polygon/}

\printbibliography

\end{document}